\documentclass[12pt, repno]{amsart}

\usepackage{setspace}

\usepackage{fancyhdr}
\usepackage{tablefootnote}

\usepackage{mathbbol}
\usepackage{amsfonts}
\usepackage{microtype}
\usepackage{amsfonts,amssymb,amsfonts,amsthm,amsmath}
\usepackage{cleveref}
\usepackage{mathrsfs}
\usepackage{enumerate}
\usepackage{graphicx,,epsfig,a4,url}
\usepackage[all]{xy}
\usepackage{color,xcolor}
\usepackage{comment}
\usepackage{tikz-cd}

\newtheorem{thm}{Theorem}[section]

\newtheorem{lem}[thm]{Lemma}

\newtheorem{prop}[thm]{Proposition}
\theoremstyle{definition}
\newtheorem{defn}[thm]{Definition}

\newtheorem{conj}[thm]{Conjecture}
\theoremstyle{remark}
\newtheorem{rem}[thm]{Remark}

\numberwithin{equation}{section}

\theoremstyle{definition}

\title{The equivariant coarse Baum--Connes conjecture for metric spaces with proper group actions}
\author{Jintao Deng \and Benyin Fu \and Qin Wang}
\thanks{Supported in part by NSFC (No. 11871342, 11771143, 11831006, 12171156). }
\date{}

\begin{document}

\maketitle

\begin{abstract}
The equivariant coarse Baum--Connes conjecture interpolates between the Baum--Connes conjecture for a discrete group and the coarse Baum--Connes conjecture for a proper metric space. In this paper, we study this conjecture under certain assumptions. More precisely, assume that a countable discrete group $\Gamma$ acts properly and isometrically on a discrete metric space $X$ with bounded geometry,  not necessarily cocompact. We show that if the quotient space $X/\Gamma$ admits a coarse embedding into Hilbert space and $\Gamma$ is amenable, and that the $\Gamma$-orbits in $X$ are uniformly equivariantly coarsely equivalent to each other, then the equivariant coarse Baum--Connes conjecture holds for $(X, \Gamma)$. Along the way, we prove a $K$-theoretic amenability statement for the $\Gamma$-space $X$ under the same assumptions as above, namely,  the canonical quotient map from the maximal equivariant  Roe algebra of $X$ to the reduced equivariant Roe algebra of $X$ induces an isomorphism on $K$-theory.

\end{abstract}

\section{Introduction}

The Baum--Connes conjecture \cite{BC, BC94} provides an algorithm to compute the $K$-theory of reduced group $C^*$-algebras, which has  important applications in geometry, topology and analysis (see \cite{BC} for a survey). It has been verified for a large class of groups including a-T-menable groups \cite{HK} and hyperbolic groups \cite{Lafforgue}.

The coarse Baum--Connes conjecture \cite{Roe93, Ro} is a geometric analogue of the Baum--Connes conjecture, which also has significant applications in geometry and topology, such as the Novikov conjecture and Gromov's conjecture about Riemannian metric of positive scalar curvature (\cite{Ferry-W, HR, Yu06, Yu19}). Many results (c.f. \cite{CWY, Deng, Fukaya, GongWangYu, Sherry, KY06, KY, OyonoYu, SW, Ruffus1, Ruffus2}) have been achieved in recent years after Yu's breakthrough to the coarse Baum--Connes conjecture for metric spaces which are coarsely embeddable into Hilbert space \cite{Yu}.

Let $X$ be a proper metric space with bounded geometry, and let $\Gamma$ be a countable discrete group. Assume that $\Gamma$ acts on $X$ properly and isometrically, not necessarily cocompact. In this case, we call $X$ a {\it $\Gamma$-space}. There is an equivariant higher index map (\cite{Fu, Shan})
$${\rm Ind}^\Gamma:\lim\limits_{d\rightarrow \infty}K_*^{\Gamma}(P_d(X)) \to K_*(C^*(X)^{\Gamma}),$$
where
$K_*^{\Gamma}(P_d(X))$ is the $\Gamma$-equivariant $K$-homology group of the Rips complex $P_d(X)$ of $X$ on scale $d>0$, and $K_*(C^*(X)^{\Gamma})$ is the $K$-theory group of the equivariant Roe algebra of the $\Gamma$-space $X$. The equivariant coarse Baum--Connes conjecture states that the equivariant higher index map ${\rm Ind}^\Gamma$ above is an isomorphism provided that $X$ has bounded geometry. When the $\Gamma$-action is cocompact, $C^*(X)^\Gamma$ is Morita equivalent to the reduced group $C^*$-algebra of $\Gamma$, so that  the equivariant coarse Baum--Connes conjecture is a reformulation of  the Baum--Connes conjecture (cf. \cite{WY-book}). When the group is trivial, the equivariant coarse Baum--Connes conjecture is plainly the coarse Baum--Connes conjecture.

In \cite{Shan}, Shan proved that  the equivariant higher index map is injective when $X$ is a simply connected complete Riemannian manifold with non-positive sectional curvature and $\Gamma$ is a torsion free group acting on $X$ properly and isometrically.  In \cite{Fu}, Fu and Wang showed that the equivariant coarse Baum--Connes conjecture holds for a $\Gamma$-space $X$ with bounded geometry which admits an equivariant coarse embedding into Hilbert space. In \cite{FWY}, Fu, Wang and Yu proved that if a discrete group $\Gamma$ acts properly and isometrically on a space $X$ with bounded geometry, such that both $X/\Gamma$ and $\Gamma$ admit a coarse embedding into Hilbert space, and that the action has bounded distortion,  then the equivariant higher index map is injective for the $\Gamma$-space $X$.
\par
In a recent work \cite{AT}, Arzhantseva and Tessera answer in the negative the following well-known question \cite{DG, GK}: does coarse embeddability into Hilbert space is preserved under group extensions of finitely generated groups? Their constructions also provide the first example of a finitely generated group which does not coarsely embed into Hilbert space yet does not contain any weakly embedded expander, answering in the affirmative another open problem \cite{AT,NY}.
Their examples are a certain restricted permutational wreath products
$$\mathbb{Z}_2\wr_G H:=\left( \bigoplus_G \mathbb{Z}_2\right)\rtimes H$$
and
$$\mathbb{Z}_2\wr_G (H\times \mathbb{F}_n):=\left( \bigoplus_G \mathbb{Z}_2\right)\rtimes \Big(H\times \mathbb{F}_n\Big),$$
where $G$ is a {\em Gromov monster group}, i.e., a finitely generated group which contains an isometrically embedded expander in its Cayley group, $H$ is a {\em Haagerup monster group}, i.e., a finitely generated group with the Haagerup property but without Yu's property A, and $\mathbb{F}_n$ is a finitely generated free group. Note that both examples are {\em ``abelian-by-Haagerup"} group extensions. The reason why they do not coarsely embed into Hilbert relies on the fact that both groups contain a certain {\em relative expanders} \cite{AT}. If we take
$X=\mathbb{Z}_2\wr_G H$ or $\mathbb{Z}_2\wr_G (H\times \mathbb{F}_n)$, and $\Gamma=\bigoplus_G \mathbb{Z}_2$, then we come across with a situation that $\Gamma$ is a torsion group, the $\Gamma$-action on $X$ does not have bounded distortion, and that the space $X$ does not even  coarsely embed into Hilbert space, let alone admits a $\Gamma$-equivariant coarse embedding into Hilbert space, namely, a situation which does not satisfy the assumptions in each of the main results in \cite{Shan, Fu, FWY} mentioned above.
\par
In this paper, as motivated by these examples by Arzhantseva and Tessera, we shall further investigate the equivariant coarse Baum--Connes conjecture along this line.
The following is our main result.

\begin{thm}\label{main-thm}
Let $\Gamma$ be a countable discrete group acting properly and isometrically on a discrete metric space $X$ with bounded geometry, not necessarily co-compact. Assume that all $\Gamma$-orbits in $X$ are uniformly equivariantly coarsely equivalent. If the quotient space $X/\Gamma$ admits a coarse embedding into Hilbert space and $\Gamma$ is amenable, then the equivariant higher index map
$${\rm Ind}^\Gamma:\lim\limits_{d\rightarrow \infty}K_*^{\Gamma}(P_d(X)) \to K_*(C^*(X)^{\Gamma})$$
is an isomorphism.
\end{thm}

\par
When the group $\Gamma$ is trivial, Theorem 1.1 recovers Yu's famous result in \cite{Yu}. On the other hand, when the $\Gamma$-action is cocompact, Theorem \ref{main-thm} recovers the amenable group case of the celebrated result of Higson and Kasparov in \cite{HK},  where they proved the theorem for more general a-T-menable groups. Note that the examples $X=\mathbb{Z}_2\wr_G H$ or $\mathbb{Z}_2\wr_G (H\times \mathbb{F}_n)$, and $\Gamma=\bigoplus_G \mathbb{Z}_2$ by Arzhantseva and Tessera discussed above provide us with nontrivial examples of non-cocompact group actions satisfying the assumptions in Theorem 1.1. At this point, it is natural to expect that Theorem \ref{main-thm} be true for non-cocompact actions with $\Gamma$ being a-T-menable. However, technically in our approach to the main theorem, we have to establish an isomorphism between the maximal equivariant twisted Roe algebra and the reduced equivariant twisted Roe algebra, see Proposition \ref{iso-max-red-twist}, in which the amenability of $\Gamma$ plays an essential role. Nevertheless, once we get this isomorphism on the $C^*$-algebra level, we are able to further get a $K$-theoretic amenability statement for the $\Gamma$-space $X$,  in a way as the proof of Theorem \ref{main-thm} by using the idea of an infinite dimensional geometric Bott periodicity, see also \cite{Yu}. More precisely, we have the following result:

\begin{thm}\label{main-thm-2}
With the same assumptions as Theorem \ref{main-thm}, the map $$\lambda_*:K_*(C^*_{max}(X)^{\Gamma})\rightarrow K_*(C^*(X)^\Gamma)$$ is an isomorphism, where $\lambda_*$ is the homomorphism induced by the canonical quotient map $\lambda$ from the maximal equivariant  Roe algebra $C^*_{max}(X)^{\Gamma}$ to the reduced equivariant Roe algebra $C^*(X)^{\Gamma}$.
\end{thm}
\par
When the group $\Gamma$ is trivial, Theorem 1.2 recovers the main result in \cite{SW} by \v{S}pakula and Willett.  When the group action is co-compact, the statement of Theorem 1.2 actually holds for a more general class of groups (see \cite{Cun}), including all a-T-menable groups, as proved by Higson and Kasparov in the same paper \cite{HK}. Unfortunately, due to limitation of our approach, the amenability of the group $\Gamma$ is necessary in our proof to Theorem \ref{main-thm-2}. For the general case where the group action is non-cocompact,  we do not know whether the conclusion of Theorem \ref{main-thm-2} is true or not for $\Gamma$ being a-T-menable.
\par
The paper is organized as follows. In Section 2, we recall the reduced and maximal equivariant Roe algebras, and formulate the equivariant coarse Baum--Connes conjecture, for a discrete metric space with bounded geometry which admits a proper and isometric action, not necessarily co-compact, by a countable discrete group. In Section 3, we use the coarse embedding of the quotient space $X/\Gamma$ into a Hilbert space to define the reduced and maximal equivariant twisted Roe algebras, and their localization counterparts. We prove that the $K$-theory of the equivariant twisted localization algebra is isomorphic to the $K$-theory of the equivariant twisted Roe algebra. Moreover, we also prove that the maximal equivariant  twisted Roe algebra and the reduced equivariant twisted Roe algebra are isomorphic. In Section 4, we complete the proof of the main theorems by using a geometric analogue of Higson--Kasparov--Trout's infinite dimensional Bott periodicity.

\section{The equivariant coarse Baum--Connes conjecture}\label{Sect-def-CBC}
In this section, we shall first recall several notions from coarse geometry, and then define the reduced and maximal equivariant Roe algebras, so as to formulate the equivariant coarse Baum--Connes conjecture for a discrete metric space with bounded geometry which admits a proper and isometric action, not necessarily co-compact, by a countable discrete group.

\subsection{Roe algebras}
In this subsection, we shall recall the notions of the (maximal) equivariant Roe algebras.

Suppose that $Z$ is a proper metric space in the sense that every closed ball of $Z$ is compact. Let $\Gamma$ be a countable discrete group acting on $Z$ properly and isometrically. Recall that a $\Gamma$-action on $Z$ is $proper$ if the map
$$Z\times \Gamma\rightarrow Z\times Z, \qquad (x,\gamma)\mapsto (\gamma x,x)$$
is proper, i.e., the preimage of a compact set is compact. The group $\Gamma$ acts on $Z$ {\it isometrically} if $d(x,y)=d(\gamma x,\gamma y)$ for all $\gamma\in\Gamma$ and $x,y\in Z$. In this case, we call $Z$ a {\it $\Gamma$-space}.

Recall that a $Z$-$module$ is a separable Hilbert space $H$ equipped with a faithful and non-degenerate $*$-representation $\phi:C_0(Z)\rightarrow B(H)$ whose range contains no nonzero compact operators, where $C_0(Z)$ is
the algebra of all complex-valued continuous functions on $Z$ vanishing at infinity.

Let $Z$ be a $\Gamma$-space. we define a $\Gamma$-action on $C_0(Z)$ by
$$\gamma(f)(x)=f(\gamma^{-1}x)$$
for all $\gamma \in \Gamma$ and $f\in C_0(Z)$.

\begin{defn}
Let $H$ be an $Z$-module. We say that $H$ is a {\it covariant $Z$-module} if it is equipped with a
unitary action $\rho$ of $\Gamma$, i.e., $\rho:\Gamma\rightarrow \mathcal {U}(H)$ is a group
homomorphism from $\Gamma$ to the set of unitary elements in $B(H)$, compatible with the action
of $\Gamma$ on $Z$ in the sense that for $v\in H$, $T\in B(H)$ and $\gamma\in \Gamma$,
$$\gamma(T)(v)=\rho(\gamma)T\rho(\gamma)^*(v).$$
We call such a triple $(C_0(Z), \Gamma, \phi)$ a $covariant\ system$.
\end{defn}

\begin{defn}[\cite{KY}]
\label{d:ad_rep}
A covariant system $(C_0(Z), \Gamma, \phi)$ is $admissible$ if there exists a $\Gamma$-Hilbert space $H_Z$ and a separable and infinite-dimensional $\Gamma$-Hilbert space $H'$ such that
\begin{enumerate}
\item[(1)] $H$ is isomorphic to $H_Z\otimes H'$;
\item[(2)] $\varphi=\varphi_0\otimes I$ for some $\Gamma$-equivariant $*$-homomorphism $\varphi_0$ from $C_0(Z)$ to $B(H_Z)$ such that $\varphi_0(f)$ is not in $K(H_Z)$ for any nonzero function $f\in C_0(Z)$ and $\varphi_0$ is non-degenerate in the sense that $\{\varphi_0(f)H_Z:f\in C_0(Z)\}$ is dense in $H_Z$;
\item[(3)] for each finite subgroup $F\subset \Gamma$ and every $F$-invariant Borel subset $E$ of $Z$, there exists a trivial $F$-representation $H_E$ such that $$\chi_E H'\cong l^2(F)\otimes H_E,$$
where $\ell^2(F)$ is endowed with the left regular representation of $F$.
\end{enumerate}

\end{defn}

\begin{defn}
Let $(C_0(Z),\Gamma, \phi)$ be an admissible covariant system.
\begin{enumerate}
\item[(1)] The $support$ $\text{Supp}(T)$ of a bounded linear operator $T:H \to H$ is the complement of the set of points $(x,y)\in Z \times Z$ for which there exist $f$ and $f'$ in $C_0(Z)$ such that
$$\phi(f')T\phi(f)=0,\quad f(x)\neq 0,\quad f'(y) \neq 0.$$
\item[(2)] A bounded operator $T:H\rightarrow H$ has finite propagation if $$\mbox{propagation}(T):=\text{sup}\{d(x,y):(x,y)\in \text{Supp}(T)\}<\infty.$$
This number is called the {\it propagation} of $T$.
\item[(3)] A bounded operator $T:H\rightarrow H$ is {\it locally compact} if the operators $\phi(f)T$ and $T\phi(f)$ are compact for all $f\in C_0(Z)$.
\item[(4)] A bounded operator $T:H\rightarrow H$ is {\it $\Gamma$-invariant} if $\gamma(T)=T$ for all $\gamma\in \Gamma$.
\end{enumerate}
\end{defn}

Now we are ready to define the equivariant Roe algebra.
\begin{defn}
Let $(C_0(Z),\Gamma, \phi)$ be an admissible covariant system. The {\it algebraic equivariant Roe algebra} $\mathbb{C}[Z]^{\Gamma}$ is the $*$-subalgebra of $B(H)$ consisting of all locally compact and $\Gamma$-invariant operators with finite propagation, where $B(H)$ is the algebra of all bounded operators on $H$. The {\it equivariant Roe algebra} $C^*(Z)^{\Gamma}$ is the closure of the $*$-algebra $\mathbb{C}[Z]^{\Gamma}$ under the operator norm in $B(H)$.
\end{defn}

A proper $\Gamma$-space $Z$ is said to have {\it bounded geometry} if there exists a $\Gamma$-invariant subset $X \subset Z$ such that
\begin{enumerate}[(1)]
    \item $X$ is discrete and $\#B_R(x)<\infty$ for all $r>0$, where $\#B_R(x)$ is the number of the elements in $\{z\in X: d(x,z)\leq R\}$;
    \item $X$ is a $c$-net in $Z$ for some $c>0$ in the sense that $Z=N_c(X)=\{z\in Z: d(z, X)\leq c\}$.
\end{enumerate}

The $*$-algebra $\mathbb{C}[Z]^\Gamma$ can be equipped with a $C^*$-norm associated with any $*$-representation $\phi:\mathbb{C}[Z]^\Gamma \to H_{\phi}$. To define the maximal norm $\mathbb{C}[Z]^\Gamma$, we need he following result which is essentially proved in \cite{GongWangYu}.

\begin{lem}\label{lem-maximal-norm}
Let $Z$ be a proper $\Gamma$-space with bounded geometry. For any $T\in \mathbb{C}[Z]^\Gamma$, there exists a constant $C_T>0$ such that for any $*$-representation $\phi$ of $\mathbb{C}[Z]^\Gamma$ on a Hilbert space $\mathcal{H}_{\phi}$, we have $$\|\phi(T)\|_{B(\mathcal{H}_{\phi})}\leq C_T.$$
\end{lem}

It follows from the above lemma that the maximal norm on the $*$-algebra $\mathbb{C}[Z]^\Gamma$ is well-defined.

\begin{defn}[\cite{GongWangYu}]
Let $Z$ be a proper $\Gamma$-space. The maximal equivariant Roe algebra, denoted by $C^*_{max}(Z)^{\Gamma}$, is the completion of $\mathbb{C}[Z]^\Gamma$ with respect to the $C^*$-norm
$$\|T\|_{max}:=\sup\{\|\phi(T)\|_{B(\mathcal{H}_{\phi})},\ \phi:\mathbb{C}[Z]^\Gamma\rightarrow B(\mathcal{H}_{\phi})\ \text{is a}\ \text{$\ast$-representation}\}.$$
\end{defn}

We remark that the definition of the (maximal) equivariant Roe algebra does not depends on the choice of admissible covariant systems. The (maximal) equivariant Roe algebra $C^*(Z)^{\Gamma}$ is Morita equivalent to the (maximal) reduced group $C^*$-algebra $C_r^*(\Gamma)$ when the quotient space $Z/\Gamma$ is compact, i.e., $Z$ admits a proper and cocompact $\Gamma$-action(c.f. {\cite[Lemma~5.14]{Ro}}). Moreover, Roe algebras is invariant under the coarse equivalence in the sense of the following.
\begin{defn}\label{sec:1:def:1}
Let $Z$ and $Z'$ be proper $\Gamma$-spaces. A Borel map $f:Z\rightarrow Z'$ is said to be an \emph{equivariantly coarse embedding} if there exist non-decreasing functions $\rho_+, \rho_-: \mathbb{R}_+:=[0,+\infty) \to \mathbb{R}_+$ such that
\begin{enumerate}[(1)]
\item $\lim\limits_{r\rightarrow +\infty}\rho_{\pm}(r)\rightarrow +\infty$;
\item $\rho_-(d(x,y))\leq d(f(x),f(y))\leq \rho_+(d(x,y))$, for all $x, y\in Z$;
\item $f(\gamma x)=\gamma f(x)$ for all $x\in Z$ and $\gamma\in \Gamma$.
\end{enumerate}
\end{defn}
Moreover, the $\Gamma$-spaces $Z$ and $Z'$ are called {\it equivariantly coarsely  equivalent} if there is a equivariant coarse embedding $f: Z\to Z'$ such that the image $f(Z)$ is \emph{coarsely dense} in $Z'$ in the sense that $Z'=N_S(f(Z))$ for some $S>0$, where $N_S(f(Z))=\{y\in Y: d(y,f(Z))<S\}$. We say that a family of metric spaces $\{Z_i\}_{i\in I}$ are said to be \emph{uniformly equivariantly coarsely equivalent} if any two metric spaces $Z_i$ and $Z_j$ are equivariantly coarsely equivalent with the same functions $\rho_+, \rho_-$ and constant $S>0$ in the above definitions.

Let $H$ be a separable infinite-dimensional Hilbert space. A map $f: Z \to H$ is called a coarse embedding into Hilbert space if
if there exist non-decreasing functions $\rho_+, \rho_-: \mathbb{R}_+:=[0,+\infty) \to \mathbb{R}_+$ such that
\begin{enumerate}[(1)]
\item $\lim\limits_{r\rightarrow +\infty}\rho_{\pm}(r)\rightarrow +\infty$;
\item $\rho_-(d(x,y))\leq d(f(x),f(y))\leq \rho_+(d(x,y))$, for all $x, y\in Z$.
\end{enumerate}

The equivariant Roe algebras are invariant under equivariantly coarse equivalence.
\begin{prop}
Let $Z$ and $Z'$ be proper $\Gamma$-spaces. If there exists an equivariantly coarse equivalence $f: Z \to Z'$, then
$$C^*(Z)^{\Gamma}\cong C^*(Z')^{\Gamma},$$
and
$$C_{max}^*(Z)^{\Gamma}\cong C_{max}^*(Z')^{\Gamma}.$$

\end{prop}

By the universality of the maximal norm, the identity map on $\mathbb{C}[Z]^\Gamma$ extends to a canonical quotient map $$\lambda:C^*_{max}(Z)^{\Gamma}\rightarrow C^*(Z)^{\Gamma}.$$
 We then obtain a homomorphism
 $$\lambda_*:K_*(C^*_{max}(Z)^{\Gamma})\rightarrow K_*(C^*(Z)^{\Gamma})$$
 induced by the map $\lambda$ at the $K$-theory level.

\subsection{The assembly map}
Let us briefly recall the definition of equivariant $K$-homology introduced by Kasparov~\cite{Ka,Yu95-2} for a covariant system $(C_0(Z),\Gamma, \phi)$.

\begin{defn}
For $i=0,1$, the $K$-homology groups $K_i^\Gamma(Z)=KK_i^\Gamma(C_0(Z),\mathbb{C})$ ($i=0,1$) are generated by certain cycles module a certain equivalent relation:

\begin{enumerate}[(1)]
\item each cycle for $K_0^\Gamma(Z)$ is a triple $(H,\phi,F)$, where $\phi:C_0(Z)\to B(H)$ is a covariant $*$-representation and $F$ is a $\Gamma$-invariant bounded operator on $H$ such that $\phi(f) F-F\phi(f)$, $\phi(f)(FF^*-I)$ and $\phi(f)(F^*F-I)$ are compact operators for all $f\in C_{0}(Z), \gamma\in \Gamma$;
\item each cycle for $K_1^\Gamma(Z)$ is a triple $(H,\phi,F)$, where$\phi:C_0(Z)\to B(H)$ is a covariant $*$-representation and $F$ is a $\Gamma$-invariant bounded operator on $H$, and $F$ is a $\Gamma$-invariant and self-adjoint operator on $H$ such that $\phi(f)(F^2-I)$, and $\phi(f) F-F\phi(f)$ are compact for all $f\in C_0(Z), \gamma\in \Gamma$.
\end{enumerate}

In both cases, the equivalence relation on cycles is given by homotopy of the operator $F$. 
\end{defn}
Let us now define the {\it equivariant index map}
$${\rm Ind}^\Gamma: K^\Gamma_*(Z)\to K_*(C^*(Z)^\Gamma).$$
Note that every class in $K^{\Gamma}_*(Z)$ can be represented by a cycle $(H,\phi,F)$ such that $(C_0(Z),\Gamma,\phi)$ is an admissible covariant system~\cite{KY}.
  Let $(H,\phi, F)$ be a cycle in $K_0^\Gamma(Z)$ such that $(C_0(Z),\Gamma,\phi)$ is an admissible covariant system.
   
Let $\{U_i\}_{i\in I}$ be a locally finite, $\Gamma$-equivariant and uniformly bounded open cover of $X$, and let $\{\psi_i\}_{i\in I}$ be a $\Gamma$-equivariant partition of unity subordinate to $\{U_i\}_{i\in I}$. Define $$F'=\sum\limits_{i\in I}\phi(\sqrt{\psi_i})F\phi(\sqrt{\psi_i}),$$
where the infinite sum converges in strong topology. Note that $(H,\phi, F')$ is equivalent to $(H, \phi, F)$ in $K^\Gamma_0(Z)$. Moreover, the operator $F'$ has finite propagation, therefore $F'$ is a multiplier of $C^*(Z)^\Gamma$ and is invertible modulo $C^*(Z)^\Gamma$. Thus, $F'$ gives rise to an element in $K_0(C^*(Z)^\Gamma)$, denoted by $\partial([F'])$. We define the equivariant index map
$${\rm Ind}^\Gamma: K^\Gamma_0(Z)\to K_0(C^*(Z)^\Gamma)$$
by
$$
{\rm Ind}^\Gamma([H,\phi, F])=\partial([F'])
$$
for all $[H,\phi, F]\in K^\Gamma_0(Z)$, where $F'$ is defined as above.
Similarly, we can define the equivariant index map 
$${\rm Ind}^\Gamma: K^\Gamma_1(Z)\to K_1(C^*(Z)^\Gamma).$$
\subsection{The equivariant coarse Baum--Connes conjecture}
In this subsection, we shall recall the equivariant coarse Baum--Connes conjecture for a discrete metric space with bounded geometry.

A metric space $X$ is said to have bounded geometry if 
$$\sup_{x \in X} \#B_R(x) <\infty$$
 for each $R>0$, where $B_R(x)=\{y \in X: d(x, y)\leq R\}$.

\begin{defn}
Let $d>0$. The Rips complex $P_d(X)$ at scale $d$ is the simplicial polyhedron whose vertex set is $X$ and where a finite subset $Y\subseteq X$ spans a simplex if and only if $d(x,y)\leq d$ for all $x, y \in Y$.
\end{defn}
Let $X$ be a discrete $\Gamma$-space. Each Rips complex $P_d(X)$ ie quipped with the spherical metric. Recall that on each path connected component of $P_d(X)$, the spherical metric is the maximal metric whose restriction to the simplex spanned by a finite subset $\left\{ x_0,x_1,\cdots, x_n\right\}$ is the metric obtained by identifying the simplex with
$$S_+^n:=\Big\{(s_0,s_1,\cdots, s_n)\in \mathbb{R}^{n+1}:s_i\geq 0, \sum\limits_{i=0}^n s_i^2 =1\Big\}$$
via the map
$$\sum\limits_{i=0}^n t_i x_i\mapsto \left(\frac{t_0}{\sqrt{\sum_{i=0}^n t_i^2}}, \frac{t_1}{\sqrt{\sum_{i=0}^n t_i^2}},\cdots,\frac{t_n}{\sqrt{\sum_{i=0}^n t_i^2}}\right),$$
where $S_+^n$ is endowed with the standard Riemannian metric. The distance of a pair of points in different connected components of $P_d(X)$ is defined to be infinity.

Each Rips complex $P_d(X)$ admits a proper $\Gamma$-action by isometries defined by
$$\gamma \cdot \sum\limits_{i=1}^k c_i x_i=\sum\limits_{i=1}^k c_i( \gamma\cdot x_i)$$
for all $\sum\limits_{i=1}^k c_i x_i \in P_d(X)$ and $\gamma\in \Gamma$. 

The following conjecture is called \emph{the equivariant coarse Baum--Connes conjecture}:
\begin{conj}[The equivariant coarse Baum--Connes conjecture]
Let $X$ be a discrete metric space with bounded geometry and let $\Gamma$ be a countable discrete group acting on $X$ properly by isometries. Then the equivariant index map
$$\rm{Ind}^\Gamma:\lim\limits_{d\rightarrow\infty}K_*^\Gamma(P_d(X))\rightarrow \lim\limits_{d\rightarrow\infty}K_*(C^*(P_d(X))^\Gamma)\cong K_*(C^*(X)^\Gamma)$$
 is an isomorphism.
\end{conj}

\subsection{Localization algebras and local index map}
In this subsection,  we shall recall the notions of the localization algebras and the local index map.
\begin{defn}
Let $Z$ be a proper $\Gamma$-space. The {\it equivariant localization algebra} $C_L^*(Z)^\Gamma$ is the norm closure of the algebra of all bounded and uniformly norm-continuous functions $$f:[0,+\infty)\rightarrow C^*(Z)^\Gamma,$$
such that
$$\text{propagation}(f(t))\rightarrow 0\ \text{as}\ t\rightarrow\infty,$$
with respect to the norm
 $$\|f\|=\sup\limits_{t\in [0,+\infty)}\|f(t)\|.$$
\end{defn}

Now, let us recall the local $\Gamma$-index map
$${\rm Ind}: K_*^\Gamma(Z)\to K_*(C^*_L(Z)^\Gamma).$$
For each positive integer $n$, let $\{U_{n,i}\}_{i \in I}$ be a locally finite and $\Gamma$-invariant open cover for $Z$ such that 
$${\rm diam}(U_{n,i})< \frac{1}{n}$$ 
for all $i\in I$. Let $\{\phi_{n,i}\}_{I}$ be a continuous partition of unity subordinate to $\{U_{n,i}\}_{I}$. Let $[(H, \phi, F)]\in K_{0}^{\Gamma}(Z)$. Define a continuous path of operators $F(t)_{t\in[0, \infty)}$ on $H$ by
$$
F(t)=\sum_{i}\bigl((1-(t-n))\phi_{n,i}^{\frac{1}{2}}F\phi_{n+1,i}^{\frac{1}{2}}+(t-n)\phi_{n+2,i}^{\frac{1}{2}}F\phi_{n+1,i}^{\frac{1}{2}}\bigr)
$$
for all $t\in [n,n+1)$, where the infinite sum converges in the strong operator topology. Notice that ${\rm propagation}(F(t))\rightarrow 0$ as $t\rightarrow\infty.$
Therefore, the path $(F(t))_{t \in [0, \infty)}$ is a multiplier of $C^*_L(X)^\Gamma$ and $F(t)$ is a unitary modulo $C^*_L(Z)^\Gamma$. Hence $F(t)$ gives rise to an element $\partial[F(t)]$ in $K_0(C^*_L(Z)^\Gamma)$. We  define the local index map
$$Ind^{\Gamma}_{L}:K_0^\Gamma(Z) \to K_0(C^*_L(Z)^\Gamma).$$
Similarly, we can define
$${\rm Ind}: K_1^\Gamma(Z)\to K_1(C^*_L(Z)^\Gamma).$$

Now we have the following result which is an equivariant analogue of Theorem 3.2 in~\cite{Yu97}. It can be proved by the 

\begin{thm}\label{bcc}
Let $X$ be a discrete metric space with bounded geometry and let $\Gamma$
be a countable discrete group acting on $X$ properly and isometrically. Then the local index map
$$\rm{Ind}^\Gamma_ L: K_*^{\Gamma}(P_d(X))\rightarrow K_*(C_L^*(P_d(X))^{\Gamma})$$ is an isomorphism.
\end{thm}

For any $d>0$, there is a natural evaluation map
$$e:C_L^*(P_d(X))^{\Gamma} \to C^*(P_d(X))^{\Gamma}$$
by
$$e(f)=f(0)$$
for all $f\in C_L^*(P_d(X))^{\Gamma}$. This is a $*$-homomorphism, thus induces a homomorphism
$$e_*:K_*(C_L^*(P_d(X))^{\Gamma}) \to K_*(C^*(P_d(X))^{\Gamma})$$
on $K$-theory.

We have the following commutative diagram:
$$
\xymatrix{
                &         \lim\limits_{d\rightarrow \infty}K_*(C_L^*(P_d(X))^{\Gamma}) \ar[d]^{e_*}     \\
  \lim\limits_{d\rightarrow \infty}K_*^{\Gamma}(P_d(X)) \ar[ur]_{\cong}^{\rm{Ind}_L^{\Gamma}} \ar[r]^{\rm{Ind}^{\Gamma}} & \lim\limits_{d\rightarrow \infty}K_*(C^*(P_d(X))^{\Gamma}).            }
$$

As a result, in order to prove that the map $\rm{Ind}^\Gamma$ is an isomorphism, it suffices to show that the map
$$e_*:\lim\limits_{d\rightarrow \infty}K_*(C_L^*(P_d(X))^{\Gamma})\rightarrow \lim\limits_{d\rightarrow \infty}K_*(C^*(P_d(X))^{\Gamma})$$
induced by the evalaution map on $K$-theory is an isomorphism.

\section{Equivariant twisted Roe algebras and localization algebras}

In this section, we shall follow the constructions in \cite{Yu} to define $C^*$-algebras encoding the geometry of the quotient space $X/\Gamma$. These algebras include maximal and reduced equivariant twisted Roe algebras, and equivariant twisted localization algebras. We prove that the K-theory of equivariant twisted localization algebras is isomorphic to the K-theory of equivariant twisted Roe algebras under the assumptions on the geometry of $X/\Gamma$ and the orbits. We also prove that the canonical quotient map between the maximal equivariant twisted Roe algebras and the reduced equivariant twisted Roe algebras is an isomorphism under the same assumptions.

\subsection{The equivariant twisted Roe algebras along the quotient space}\label{twist-quotient}

Let $\Gamma$ be a countable discrete group, and let $X$ be a $\Gamma$-space with bounded geometry. Assume that $\xi: X/\Gamma \to H$ is a coarse embedding of $X$ into  Hilbert space $H$. Let $\pi:X\rightarrow X/\Gamma$ be the quotient map.
We shall first recall the $C^*$-algebra associated to an infinite-dimensional Euclidean space introduced by Higson--Kasparov--Trout \cite{HKT}.

Let $H$ be a real (countably infinite-dimensional) Hilbert space.  Denote by $V_a$, $V_b$, etc., the finite-dimensional affine subspaces of $H$. Let $V_a^0$
be the finite-dimensional linear subspaces of $H$ consisting of differences of elements
of $V_a$. Let ${\rm Cliff}(V_a^0)$ be the complexified Clifford algebra of $V_a^0$, and let $\mathcal {C}(V_a)$ be the graded $C^*$-algebra of continuous functions  from $V_a$ to ${\rm Cliff}(V_a^0)$ vanishing at infinity. Let $\mathcal {S}$ be the $C^*$-algebra $C_0(\mathbb{R})$, graded according to
the odd and even functions. Define the graded tensor product $$\mathcal {A}(V_a)=\mathcal {S}\widehat{\otimes}\mathcal {C}(V_a).$$
If $V_a\subseteq V_b$, we have a decomposition $V_b=V_{ba}^0+V_a$, where $V_{ba}^0$ is the orthogonal complement of $V_a^0$ in $V_b^0$. For each $v_b \in V_b$, there is a unique decomposition $v_b=v_{ba}+v_a$, where $v_{ba}\in V_{ba}^0$ and $v_a\in V_a$. Every function $h$ on $V_a$ can be extended to a function $\tilde{h}$ on $V_b$ by the formula $$\tilde{h}(v_{ba}+v_a)=h(v_a).$$

If $V_a\subseteq V_b$, we use $C_{ba}$  to denote the function $V_b\rightarrow {\rm Cliff}(V_b^0)$, $v_b\mapsto v_{ba}$, where $v_{ba}$ is
considered as an element of ${\rm Cliff}(V_b^0)$ via the inclusion $V_{ba}^0\subset \textrm{Cliff}(V_b^0)$.
Let $X$ be the unbounded multiplier of $\mathcal{S}$  with degree one given by the point-wise multiplication with the function $x\mapsto x$. Note that $X\widehat{\otimes}1+1\widehat{\otimes}C_{ba}$ is an unbouded, essentially self-adjoint operator with degree one \cite{HKT}.
So we can use functional calculus to define a $*$-homomorphism
$$\beta_{ba}:\mathcal {A}(V_a)\rightarrow \mathcal {A}(V_b)$$
by the formula
$$\beta_{ba}(g\widehat{\otimes}h)=g(X\widehat{\otimes}1+1\widehat{\otimes}C_{ba})(1\widehat{\otimes}\tilde{h}),$$
for $g\in \mathcal {S}$ and $h\in \mathcal {C}(V_a)$, where $g(X\widehat{\otimes}1+1\widehat{\otimes}C_{ba}).$

\begin{rem}
Let $g_0(x)=e^{-x^2}$ and $g_1(x)=xe^{-x^2}$. It is not difficult to check that $$g_0(X\widehat{\otimes}1+1\widehat{\otimes}C_{ba})=g_0(X)\widehat{\otimes}g_0(C_{ba})$$ for $g_0=e^{-x^2}$ and $$g_1(X\widehat{\otimes}1+1\widehat{\otimes}C_{ba})=g_0(X)\widehat{\otimes}g_1(C_{ba})+g_1(X)\widehat{\otimes}g_0(C_{ba})$$ for $g_1(x)=xe^{-x^2}$.
Moreover $g_0(X)=g_0(x)$, $g_1(X)=g_1(x)$ as multiplication operators, and $(g_0(C_{ba})\tilde{h})(v_b)=g_0(\|v_{ba}\|)h(v_a)$, $(g_1(C_{ba})\tilde{h})(v_b)=v_{ba}g_0(\|v_{ba}\|)h(v_a)$ when we view $C_{ba}$ as an unbounded multiplier on $\mathcal {C}(V_b)$.
\end{rem}

The maps $(\beta_{ba})$ make the collection $(\mathcal {A}(V_a))$ into a directed system as $V_a$ ranges over finite-dimensional affine subspaces of $H$. Define the $C^*$-algebra $\mathcal {A}(H)$ by $$\mathcal {A}(H)=\lim\limits_{\rightarrow}\mathcal {A}(V_a),$$
where the inductive limit is taken over all finite-dimensional affine subspaces of $H$.

Let $\mathbb{R}_+\times H$ be the topological space endowed with the weakest topology for which the projection onto $H$ is weakly continuous and the function $(t,w)\mapsto t^2+\|w\|^2$ is continuous. This topology makes $\mathbb{R}_+\times H$ a locally compact Hausdorff space. Note that for $v\in H$ and $r>0$, the subset $$B(v,r)=\{(t,w)\in \mathbb{R}_+\times H: t^2+\|v-w\|^2<r^2\}$$
is an open subset of $\mathbb{R}_+\times H$.
For each finite-dimensional subspace $V_a\subseteq H$, $C_0(\mathbb{R}_+ \times V_a)$ is included in $\mathcal {A}(V_a)$ as its center. If $V_a\subseteq V_b$, then $\beta_{ba}$ takes $C_0(\mathbb{R}_+\times V_a)$ into $C_0(\mathbb{R}_+\times V_b)$. Then $C^*$-algebra $\lim\limits_{\rightarrow}C_0(\mathbb{R}_+\times V_a)$ is isomorphic to $C_0(\mathbb{R}_+\times H)$,
where the direct limit is over the directed set of all finite-dimensional affine subspaces $V_a$ of $H$.

\begin{defn}\label{supp_a}
The {\it support} of an element $a\in \mathcal {A}(H)$ is the complement of all $(t,v)\in \mathbb{R}_+\times H$ such that there exists $g\in C_0(\mathbb{R}_+\times H)$ with $g(t,v)\neq 0$ and $g\cdot a=0$.
\end{defn}

For each $d>0$, the $\Gamma$-action on $X$ induces a proper and isometric $\Gamma$-action on $P_d(X)$ by
$$\gamma \cdot \sum\limits_{i=1}^k c_i x_i=\sum\limits_{i=1}^k c_i( \gamma\cdot x_i)$$
for $\sum\limits_{i=1}^k c_i x_i \in P_d(X)$ and $\gamma \in \Gamma$.

Note that the coarse embedding $\xi:X/\Gamma\rightarrow H$ induces a coarse embedding $$\xi:P_d(X)/\Gamma\rightarrow H$$
by
$$\xi\Big(\pi(\sum\limits_{i=1}^k c_i x_i)\Big)=\sum\limits_{i=1}^k c_i \xi(\pi(x_i)),$$
for all $\sum\limits_{i=1}^k c_i x_i \in P_d(X)$ where $\pi: P_d(X) \to P_d(X)/\Gamma$ is the quotient map induced by the $\Gamma$-action on $P_d(X)$.

For any element $x\in P_d(X)$, we define
$$W_k(\pi(x))=\xi(\pi(x))+\,\mathrm{span}\{\xi(\pi(y))-\xi(\pi(x)) : y\in X/\Gamma, d(\pi(x),\pi(y))\leq k^2\}.$$
The subspace $W_k(\pi(x))$ is a finite-dimensional affine subspace of $H$ by the bounded geometry of $X$. We can assume that $\bigcup\limits_{k\geq 1} W_k(\pi(x))$ is dense in $H$, otherwise we can replace $H$ with the norm closure of $\bigcup\limits_{k\geq 1} W_k(\pi(x))$.
For $k<k'$, let
$$\beta_{k',k}(\pi(x)):\mathcal {A}(W_k(\pi(x)))\rightarrow \mathcal {A}(W_{k'}(\pi(x)))$$
be the map defined by the inclusion of $W_k(\pi(x))$ into $W_{k'}(\pi(x))$.

Let
$$\beta_k(\pi(x)):\mathcal {A}(W_k(\pi(x))\rightarrow \mathcal {A}(H)$$
the map defined by the definition of $\mathcal {A}(H)$. We write $\beta(\pi(x))$ for $$\beta_0(\pi(x)):\mathcal {S}\cong \mathcal {A}(W_0(\pi(x)))\rightarrow \mathcal {A}(H).$$

For each $d>0$, fix a countable and $\Gamma$-invariant dense subset $X_d$ of $P_d(X)$ such that $X_{d_1}\subseteq X_{d_2}$ if $d_1\leq d_2$.
\begin{defn}\label{alg-twist}
The \emph{algebraic twisted equivariant Roe algebra} $C_{alg}^*(P_d(X), \mathcal {A}(H))^\Gamma$ is defined to be the set of all functions $T: X_d \times X_d \to \mathcal{A}(H)\widehat{\otimes}K$ such that
\begin{enumerate}[(1)]
\item There exists $M\geq 0$ such that $\|T(x,y)\|\leq M$ for all $x,y\in X$;

\item There exists an integer $N$ such that
$$T(x,y)\in \beta_N(\pi(x))(\mathcal {A}(W_N(\pi(x))))\widehat{\otimes}K\subseteq \mathcal {A}(H)\widehat{\otimes}K$$ for all $x,y\in X$, where $\beta_N(\pi(x)):\mathcal {A}(W_N(\pi(x)))\rightarrow \mathcal {A}(H)$ is the $*$-homomorphism associated to the inclusion of $W_N(\pi(x))$ into $H$, and $K$ is the algebra of compact operators on the Hilbert space $H_0\widehat{\otimes}l^2(\Gamma)$;

\item There exists $r_1>0$  such that if $d(x,y)>r_1$, then $T(x,y)=0$;

\item There exists $L>0$ such that for each $y\in X_d$, $$\sharp\{x:T(x,y)\neq 0\}\leq L,\qquad \sharp\{x:T(y,x)\neq 0\}\leq L;$$

\item For any bounded set $B\subseteq P_d(X)$, the set
$$\{(x,y)\in B\times B\cap X_d\times X_d|T(x,y)\neq 0\}$$
is finite;

\item There exists $r_2>0$ such that
$$
\mbox{Supp}(T(x,y)) \subseteq \left\{(s,h)\in\mathbb{R}_+\times H:s^2+\|h-\xi(\pi(x))\|^2<r_2^2\right\}
$$
for $x,y\in X$;

\item There exists $c>0$ such that if $Y=(s,h)\in \mathbb{R}_+\times W_N(\pi(x))$ with $\|Y\|=\sqrt{s^2+\|h\|^2}\leq 1$ and if $T_1(x,y)$ satisfies that
$(\beta_N(\pi(x))\widehat{\otimes}1)(T_1(x,y))=T(x,y)$, then the derivative of the
function $T_1(x,y):\mathbb{R}\times W_N(\pi(x))\rightarrow \text{Cliff}(W_N(\pi(x)))\widehat{\otimes}K$ in the direction of $Y$ exists, denoted by $D_Y(T_1(x,y))$, and $\|D_Y(T_1(x,y))\|\leq c$ for all $x,y\in X$;
\item $T(\gamma x, \gamma y)=T(x,y)$ for all $x,y\in X$ and $\gamma\in \Gamma$.
\end{enumerate}
\end{defn}

We define a multiplication on $C_{alg}^*(P_d(X), \mathcal {A}(H))^\Gamma$ by $$(T_1 T_2)(x,y)=\sum\limits_{z\in X_d}T_1(x,z)T_2(z,y)$$
and an involution by $T^*(x,y)=\big(T(y,x)\big)^*$. Then $C^*_{alg}(P_d(X),\mathcal {A}(H))^\Gamma$ is a $*$-algebra.
Let
$$E=\left\{\sum\limits_{x\in X_d}a_x[x]:a_x\in \mathcal {A}(H)\widehat{\otimes} K,\sum\limits_{x\in X_d}a_x^*a_x\ \text{converge in norm}\right\}.$$
Then $E$ is a Hilbert module over $\mathcal {A}(H)\widehat{\otimes}K$ with
$$\left\langle\sum\limits_{x\in X_d}a_x[x],\sum\limits_{x\in X_d}b_x[x]\right\rangle=\sum\limits_{x\in X_d}a_x^*b_x,$$
$$\left(\sum\limits_{x\in X_d}a_x[x]\right)a=\sum\limits_{x\in X_d}a_xa[x]$$
for all $a\in \mathcal {A}(H)\widehat{\otimes}K$ and $\sum\limits_{x\in X_d}a_x[x]\in E$.

Now we can define a $*$-representation of $C^*_{alg}(X,\mathcal {A}(H))^{\Gamma}$ on $E$ by
$$T\left(\sum\limits_{x\in X_d}a_x[x]\right)=\sum\limits_{y\in X_d}\left(\sum\limits_{x\in X_d}T(y,x)a_x\right)[y]$$
for all $T\in C^*_{alg}(P_d(X),\mathcal {A}(H))^{\Gamma}$ and $\sum\limits_{x\in X_d}a_x[x]\in E$.

\begin{defn}
The {\it reduced equivariant twisted Roe algebra} $C^*(P_d(X),\mathcal {A}(H))^{\Gamma}$ is the operator norm closure of $C^*_{alg}(P_d(X),\mathcal {A}(H))^{\Gamma}$ in $B(E)$, where $B(E)$ is the algebra of all bounded adjointable homomorphisms from $E$ to $E$.
\end{defn}

By similar arguments as in Lemma \ref{lem-maximal-norm}, the maximal norm on the algebraic equivariant twisted Roe algebra is also well-defined. Thus, we can define the maximal equivariant twisted Roe algebra.
\begin{defn}
 The {\it maximal equivariant twisted Roe algebra} $C_{max}^*(P_d(X),\mathcal {A}(H))^{\Gamma}$ is defined to be the completion of the $*$-algebra $C^*_{alg}(P_d(X),\mathcal {A}(H))^{\Gamma}$ with respect to the maximal norm
$$\|T\|_{max}:=\sup\{\|\phi(T)\|:\phi:C^*_{alg}(P_d(X),\mathcal {A}(H))^{\Gamma}\rightarrow B(E_{\phi}),\ \text{a}\ \text{$\ast$-representation}\}.$$
\end{defn}

Let $C_{L,alg}^*(P_d(X), \mathcal {A}(H))^\Gamma$ be the set of all bounded, uniformly norm continuous functions
$$g:\mathbb{R}_+\rightarrow C_{alg}^*(P_d(X),\mathcal {A}(H))^\Gamma$$
such that:
\begin{enumerate}[(1)]
  \item there exists $N$ such that
  $$g(t)(x,y)\in (\beta_N(\pi(x))\hat{\otimes}1)(\mathcal {A}(W_N(\pi(x)))\hat{\otimes}K)\subseteq \mathcal {A}(H)\hat{\otimes}K$$
  for all $x,y\in X_d$ and $t\in \mathbb{R}_+$;
  \item $\lim\limits_{t \to \infty}\mbox{propagation}(f(t))\to 0 \ \text{as}\  t\rightarrow \infty$;
  \item there exists $R>0$ such that $\text{Supp}(g(t)(x,y))\subseteq B_{\mathbb{R}_+\times H}(\xi(\pi(x)),R)$ for all $ x,y\in X_d$ and $t\in \mathbb{R}_+$;
  \item there exists $L>0$ such that for any $y\in P_d(X)$,
  \begin{center}
  $\sharp\{x:g(t)(x,y)\neq 0\}<L$, $\sharp\{x:g(t)(y,x)\neq 0\}<L$
   \end{center}
  for any $t\in \mathbb{R}_+$.
  \item There exists $c>0$ such that if $Y=(s,h)\in \mathbb{R}_+\times W_N(\pi(x))$ with $\|Y\|=\sqrt{s^2+\|h\|^2}\leq 1$ and if $g_1(t)(x,y)$ satisfies that
  $$(\beta_N(\pi(x))\widehat{\otimes}1)(g(t)_1(x,y))=T(x,y),$$
  then the derivative of the function
  $$g_1(t)(x,y):\mathbb{R}\times W_N(\pi(x))\rightarrow \mbox{Cliff}(W_N(\pi(x)))\widehat{\otimes}K$$
  in the direction of $Y$ exists, denoted by $D_Y(g(t)_1(x,y))$, and $$\|D_Y(g(t)_1(x,y))\|\leq c$$
  for all $x,y\in X$ and $t \in \mathbb{R}_+$;
\end{enumerate}

\begin{defn}\label{local-alg}
The {\it twisted equivariant localization algebra} $C_{L}^*(P_d(X),\mathcal {A}(H))^{\Gamma}$ is the completion of $C_{L,alg}^*(P_d(X),\mathcal {A}(H))^{\Gamma}$ with respect to the norm
$$\|g\|=\sup\limits_{t\in \mathbb{R}_+}\|g(t)\|_{C^*(P_d(X),\mathcal {A}(H))^\Gamma}.$$
\end{defn}

There is an evaluation map
$$e:C_L^*(P_d(X),\mathcal {A}(H))^{\Gamma}\rightarrow C^*(P_d(X),\mathcal {A}(H))^\Gamma,$$
by
$$e(g)=g(0),$$
for all $g \in C_L^*(P_d(X),\mathcal {A}(H))^{\Gamma}$.
We have a homomorphism
$$e_*:K_*(C_L^*(P_d(X),\mathcal {A}(H))^{\Gamma})\rightarrow K_*(C^*(P_d(X),\mathcal {A}(H))^\Gamma),$$
induced by the evaluation map on $K$-theory.

In the rest of this subsection, we shall prove the following result.
\begin{prop}\label{twisted-cbc}
Let $\Gamma$ be a countable discrete group acting properly and isometrically on a discrete metric space $X$ with bounded geometry. Assume that all $\Gamma$-orbits in $X$ are uniformly equivariantly coarsely equivalent. If the quotient space $X/\Gamma$ admits a coarse embedding into Hilbert space and $\Gamma$ satisfies the Baum--Connes conjecture with coefficients, then the map
$$e_*:\lim\limits_{d\rightarrow \infty}K_*(C_L^*(P_d(X),\mathcal {A}(H))^{\Gamma})\rightarrow \lim\limits_{d\rightarrow \infty}K_*(C^*(P_d(X),\mathcal {A}(H))^\Gamma)$$
is an isomorphism.
\end{prop}

To prove the above proposition, we need to analyze the ideals of the twisted Roe algebras and localization algebras associated to the open subsets of $\mathbb {R}_+\times H$. Let $O$ be an open subset of $\mathbb{R}_+\times H$. Define $C_{alg}^*(P_d(X),\mathcal {A}(H))_O^{\Gamma}$ to be a $*$-subalgebra of $C_{alg}^*(P_d(X),\mathcal {A}(H))^{\Gamma}$ by
$$C_{alg}^*(P_d(X),\mathcal {A}(H))_O^{\Gamma}=\{T\in C_{alg}^*(P_d(X),\mathcal {A}(H))^{\Gamma}: \text{Supp}(T(x,y))\subseteq O, \forall x,y\in X\}.$$
We can define $C_{max}^*(P_d(X),\mathcal {A}(H))_O^{\Gamma}$ and $C^*(P_d(X),\mathcal {A}(H))_O^{\Gamma}$ to be the norm closure of $C_{alg}^*(P_d(X),\mathcal {A}(H))_O^{\Gamma}$ in $C_{max}^*(P_d(X),\mathcal {A}(H))^{\Gamma}$ and $C^*(P_d(X),\mathcal {A}(H))^{\Gamma}$, respectively. We would like to point out that the norm on $C_{max}^*(P_d(X),\mathcal {A}(H))_O^{\Gamma}$ is not the maximal norm of the $*$-algebra $C_{alg}^*(P_d(X),\mathcal {A}(H))_O^{\Gamma}$.

Similarly, let
$$C_{L,alg}^*(P_d(X),\mathcal {A})^{\Gamma}_O:=\{g\in C_{L,alg}^*(P_d(X),\mathcal {A})^{\Gamma}:\text{Supp}(g(t))\subseteq X_d\times X_d\times O\}.$$
Define $C_{L}^*(P_d(X),\mathcal{A})^{\Gamma}_O$ to be the norm closure of $C_{L,alg}^*(P_d(X),\mathcal {A})^{\Gamma}_O$ in $C_{L}^*(P_d(X),\mathcal{A})^{\Gamma}$.

Let $$B(\xi(\pi(x)),r)=\{(s,v)\in \mathbb {R}_+\times H, s^2+\|v-\xi(\pi(x))< r^2\|\},$$
where $\pi:X \to X/\Gamma$ is the quotient map and $\xi: X/\Gamma \to H$ is the coarse embedding. Denote $O_r=\bigcup\limits_{x\in X}B(\xi(\pi(x)),r)$. We have
$$C^*(P_d(X),\mathcal {A}(H))^{\Gamma}=\lim\limits_{r\rightarrow +\infty} C^*(P_d(X),\mathcal{A}(H))_{O_r}^{\Gamma}$$
and $$C_L^*(P_d(X),\mathcal {A}(H))^{\Gamma}=\lim\limits_{r\rightarrow +\infty} C_L^*(P_d(X),\mathcal {A}(H))_{O_r}^{\Gamma}.$$

Recall that $\pi: X \to X/\Gamma$ is the natural quotient map associated to the $\Gamma$-action on $X$.

For any $r>0$, since $X/\Gamma$ has bounded geometry, there exists finitely many mutually disjoint subsets of $X/\Gamma$, say $\widetilde{X}_k$ for $1\leq k\leq k_0$, such that
\begin{enumerate}[(1)]
    \item $X/\Gamma=\bigsqcup\limits_{k=1}^{k_0}\widetilde{X_k};$
    \item for each $\widetilde{X}_k$ and for all $z,z' \in \widetilde{X}_k$, if $z\neq z'$, then
    $$
    \|\xi(z)-\xi(z')\|\geq 2r,
    $$
    where $\xi:X/\Gamma \to H$ is the coarse embedding.
\end{enumerate}
Let $Y \subset X$ be a $\Gamma$-domain of $X$ in the sense that each orbit intersects with $Y$ at exactly one point.
For simplicity, we denote
$$Y_k:=(\pi|_{Y})^{-1}(\widetilde{X}_k)\subset Y$$
for $k=1,2,\cdots, k_0$, where $\pi|_{Y}: Y \to X/\Gamma$ is the restriction of $\pi: X \to X/\Gamma$ to the $\Gamma$-domain $Y\subset X$. For each $k=1,2,\cdots, k_0$ and $r>0$, denote
$$O_{r,k}=\bigcup\limits_{x \in Y_k}B(\xi(\pi(x)), r).$$

\begin{lem} \label{decomp-twisted-cbc}
With the same assumptions as Proposition \ref{twisted-cbc} and $\{O_{r,k}\}_{r, k}$ defined above, the map $$e_*:\lim\limits_{d\rightarrow \infty}K_*(C_L^*(P_d(X),\mathcal {A}(H))_{O_{r,k}}^{\Gamma})\rightarrow \lim\limits_{d\rightarrow \infty}K_*(C^*(P_d(X),\mathcal {A}(H))_{O_{r,k}}^\Gamma)$$
induced by the evaluation map is an isomorphism for all $r>0$ and $1\leq k\leq k_0$.
\end{lem}

\begin{proof}
For each $x\in X$, set $O_x=B(\xi(\pi(x)), r)$ for brevity. Define
$$\Pi^u_{x \in Y_k}C^*_{alg}(P_d(X), \mathcal {A}(H))_{O_x}^{\Gamma}$$
to be the $*$-subalgebra of the product algebra $\Pi_{x\in Y_k}C^*_{alg}(P_d(X), \mathcal {A}(H))_{O_x}^{\Gamma},$ such that for each element
$$\Pi_{x\in Y_k} T_x\in \Pi^u_{x\in Y_k}C^*_{alg}(P_d(X), \mathcal {A}(H))_{O_x}^{\Gamma},$$
the conditions in Definition \ref{alg-twist} are satisfied for all $T_x$ with $x\in X_k$ uniformly.

For each element $a \in \mathcal{A}(H)$ with support contained in $O_{r,k}$, it can be decomposed as a sum
$$a=\sum_{x\in Y_k} a_x,$$
where each $a_x$ is a restriction of $a$ to $O_x$ and can be viewed as a function supported in $O_x$ for all $x \in X$. It is obvious that
$$C^*_{alg}(P_d(X), \mathcal {A}(H))_{O_{r,k}}^{\Gamma}=\Pi^u_{x\in Y_k}C^*_{alg}(P_d(X), \mathcal {A}(H))_{O_x}^{\Gamma}.$$
We define $\Pi^u_{x\in Y_k}C^*(P_d(X), \mathcal {A}(H))_{O_x}^{\Gamma}$ to be the norm closure of
$$
\left\{\Pi_{x\in Y_k}T_x~\Big|~T_x\in C_{alg}^*(P_d(X), \mathcal {A}(H))_{O_x}^{\Gamma}, \sup_{x\in X_k}\|T_x\|<\infty\right\}
$$
under the supremum norm.

Similarly, we define $\Pi^u_{x\in Y_k}C_L^*(P_d(X), \mathcal {A}(H))_{O_x}^{\Gamma}$ to be the $C^*$-subalgebra of
$$\left\{\Pi_{x\in Y_k}b_x~\Big| \Pi_{x\in Y_k} b_x(t)\in \Pi^u_{x\in Y_k}C_{alg}^*(P_d(X), \mathcal {A}(H))_{O_x}^{\Gamma} \ \mbox{and}\ \sup_{x\in Y_k}\|b_x\|<\infty\right\}$$
generated by elements $\Pi_{x\in Y_k}b_x$, such that
\begin{enumerate}[(1)]
    \item the function
$$ \Pi_{x\in Y_k}b_x:\mathbb{R}_+\rightarrow \Pi^u_{x\in Y_k}C_{alg}^*(P_d(X), \mathcal {A}(H))_{O_x}^{\Gamma}$$
is uniformly norm-continuous in $t\in \mathbb{R}_+$;
\item there exists a bounded function $c(t)$ on $\mathbb{R}_+$ with $\lim\limits_{t\rightarrow\infty}c(t)=0$ such that $b_x(t)(y,y')=0$ whenever $d(y,y')>c(t)$ for all $x\in Y_k,y,y'\in X_d$ and $t\in \mathbb{R}_+$.
\end{enumerate}

By the definition of the algebraic equivariant twisted Roe algebra, we can prove that $$C^*(P_d(X),\mathcal {A})^{\Gamma}_{O{r,k}}\cong \lim\limits_{S\rightarrow \infty}\Pi^u_{x\in Y_k}C^*(P_d(N_S(\Gamma\cdot x)), \mathcal {A}(H))_{O_x}^{\Gamma}$$
and $$C_L^*(P_d(X),\mathcal {A})^{\Gamma}_{O_{r,k}}\cong \lim\limits_{S\rightarrow \infty}\Pi^u_{x\in Y_k}C_L^*(P_d(N_S(\Gamma\cdot x)), \mathcal {A}(H))_{O_x}^{\Gamma},$$
where $N_S(\Gamma\cdot x)$ is the $S$-neighbourhood of $\Gamma\cdot x$ in $X$.
So
$$\lim\limits_{d\rightarrow \infty}C^*(P_d(X),\mathcal {A})^{\Gamma}_{O{r,k}}\cong \lim\limits_{d\rightarrow \infty}\lim\limits_{S\rightarrow \infty}\Pi^u_{x\in Y_k}C^*(P_d(N_S(\Gamma\cdot x)), \mathcal {A}(H))_{O_x}^{\Gamma}$$
$$\ \ \ \ \ \ \ \ \ \ \ \ \ \ \ \ \ \ \ \ \ \ \ \ \ \ \ \ \cong \lim\limits_{S\rightarrow \infty}\lim\limits_{d\rightarrow \infty}\Pi^u_{x\in Y_k}C^*(P_d(N_S(\Gamma\cdot x)), \mathcal {A}(H))_{O_x}^{\Gamma}$$
and $$\lim\limits_{d\rightarrow \infty}C_L^*(P_d(X),\mathcal {A})^{\Gamma}_{O{r,k}}\cong \lim\limits_{d\rightarrow \infty}\lim\limits_{S\rightarrow \infty}\Pi^u_{x\in Y_k}C_L^*(P_d(N_S(\Gamma\cdot x)), \mathcal {A}(H))_{O_x}^{\Gamma}$$
$$\ \ \ \ \ \ \ \ \ \ \ \ \ \ \ \ \ \ \ \ \ \ \ \ \ \ \ \ \cong \lim\limits_{S\rightarrow \infty}\lim\limits_{d\rightarrow \infty}\Pi^u_{x\in Y_k}C_L^*(P_d(N_S(\Gamma\cdot x)), \mathcal {A}(H))_{O_x}^{\Gamma}.$$

It suffices to show that
$$e_*:\lim\limits_{d\rightarrow \infty}K_*\left(\Pi^u_{x\in Y_k}C_L^*(P_d(N_S(\Gamma\cdot x)), \mathcal {A}(H))_{O_x}^{\Gamma}\right)\ \ \ \ \ \  \ \ \ \ \ \ \ \ \ \ \ \ \ \ \ \ \ \ \ \ \ \ \ \ \ \ \ \ \ \ \ \ $$
$$\ \ \ \ \ \ \ \ \ \ \ \ \ \ \ \ \ \ \ \ \ \ \ \ \ \ \ \ \ \ \ \ \ \ \to \lim\limits_{d\rightarrow \infty}K_*(\Pi^u_{x\in Y_k}C^*(P_d(N_S(\Gamma\cdot x)), \mathcal {A}(H))_{O_x}^{\Gamma})$$
induced by the evaluation map on $K$-theory is an isomorphism for each $S>0$.

For each $S>0$ and $x \in Y_k$, we define a map:
$$f_x: P_d(\Gamma) \to P_d(N_S(\Gamma\cdot x)) $$
by
$$
f_x\left(\sum c_{g}g\right)=\sum_{g} c_g g x
$$
for each $\sum_{g \in \Gamma} c_{g}g \in P_d(\Gamma)$. When $d$ is suitably large, we obtain a family of continuous maps $\{f_j\}_{j\in J_k}$ satisfying the following conditions:
\begin{enumerate}[(1)]
    \item each $f_x$ is $\Gamma$-equivariant;
    \item there exists constants $c_1, c_2>0$ such that
    $$
     c_1\cdot  d(y, y')\leq d(f_x(y),f_x(y'))\leq c_2\cdot d(x, y),
    $$
    for all $y, y' \in P_d(\Gamma)$ and $x \in Y_k$;
    \item there exists a constant $c>0$ such that $f_x(P_d(\Gamma))$ is a $c$-net in $P_d(N_S(\Gamma \cdot x))$ for all $x \in Y_k$, i.e. $P_d(N_S(\Gamma \cdot x))=N_c(f_x(P_d(\Gamma)))$, where $N_c(f_x(P_d(\Gamma)))$ is the $c$-neighborhood of $f_x(P_d(\Gamma))$ in $P_d(N_S(\Gamma \cdot x))$;
    \item $\{f_x\}_{x\in Y_k}$ is a $\Gamma$-equivariantly strongly Lipschitz homotopy equivalence.
\end{enumerate}
Following from conditions (1) and (3) above, we have that
\begin{align*}
\Pi^u_{x\in Y_k}C^*(P_d(N_S(\Gamma\cdot x)), \mathcal {A}(H))_{O_x}^{\Gamma}  \cong & \Pi^u_{x\in Y_k}C^*(P_d(\Gamma), \mathcal{A}(H))_{O_x}^{\Gamma}\\
                             \cong & C^*(P_d(\Gamma), \Pi^u_{x \in Y_k} \mathcal {A}(H)_{O_x})^{\Gamma}.
\end{align*}
For the localization algebras, following the proof of the Proposition 3.7 in \cite{Yu97} by using the storngly $\Gamma$-equivalently Lipschitz homotopy equivalence,  we also have that
\begin{align*}
\Pi^u_{x\in Y_k}C_L^*(P_d(N_S(\Gamma\cdot x)), \mathcal {A}(H))_{O_x}^{\Gamma}  \cong & \Pi^u_{x\in Y_k}C_L^*(P_d(\Gamma), \mathcal {A}(H))_{O_x}^{\Gamma}\\
                             \cong & C_L^*(P_d(\Gamma), \Pi^u_{x \in X_k} \mathcal {A}(H)_{O_x})^{\Gamma}.
\end{align*}
Since the Baum--Connes conjecture with coefficients holds for the group $\Gamma$, the assembly map
$$
e_*:\lim\limits_{d\rightarrow \infty}K_*\left(C_L^*\left(P_d(\Gamma), \Pi^u_{x \in Y_k} \mathcal {A}(H)_{O_x}\right)^{\Gamma}\right) \to
\lim\limits_{d\rightarrow\infty}K_*\left(C^*\left(P_d(\Gamma), \Pi^u_{x \in Y_k} \mathcal {A}(H)_{O_x}\right)^{\Gamma}\right)
$$
is an isomorphism.

When $d>0$ is large enough, from the above discussion, we have the following commutative diagram:
\begin{equation*}
    \begin{tikzcd}
 K_*\left(\Pi^u_{x\in Y_k}C_L^*\left(P_d(N_S(\Gamma\cdot x)), \mathcal {A}(H)\right)_{O_x}^{\Gamma}\right) \ar{r}{e_*}\ar{d}{\cong} & K_*\left(\Pi^u_{x\in Y_k}C^*\left(P_d(N_S(\Gamma\cdot x)), \mathcal {A}(H)\right)_{O_x}^{\Gamma}\right)\ar{d}{\cong}\\
 K_*\left(C_L^*\left(P_d(\Gamma), \Pi^u_{x \in Y_k} \mathcal {A}(H)_{O_x}\right)^{\Gamma}\right) \ar{r}{e_*} & K_*\left(C^*\left(P_d(\Gamma), \Pi^u_{x \in Y_k} \mathcal {A}(H)_{O_x}\right)^{\Gamma}\right).
    \end{tikzcd}
\end{equation*}
It follows that the assembly map
$$e_*:\lim\limits_{d\rightarrow \infty}K_*\left(\Pi^u_{x\in Y_k}C_L^*\left(P_d(N_S(\Gamma\cdot x)), \mathcal {A}(H)\right)_{O_x}^{\Gamma}\right) \to\ \ \ \ \ \ \ \ \  \ \ \ \ \ \ \ \ \ \ \ \ \ \ \ \ \ \ \ \ \ \ \ \ \ \ \ \ \ \ \ \ $$
$$\ \ \ \ \ \ \ \ \ \ \ \ \ \ \ \ \ \ \ \ \ \ \ \ \ \ \ \ \ \ \ \ \ \ \
\lim\limits_{d\rightarrow \infty}K_*\left(\Pi^u_{x\in Y_k}C^*\left(P_d(N_S(\Gamma\cdot x)), \mathcal {A}(H)\right)_{O_x}^{\Gamma}\right)$$
is an isomorphism.
\end{proof}
To prove proposition \ref{twisted-cbc}, we also need the following lemma, which can be proved by the same arguments in Lemma 6.3 in \cite{Yu}.

\begin{lem}\label{algebra-m-v}
Let $X_1$ and $X_2$ be $\Gamma$-invariant subsets of $X$, and let $$O_r^1=\cup_{x\in X_1}B_{\mathbb{R}_+\times H}(\xi(\pi(x)),r),\quad  O_r^2=\cup_{x\in X_2}B_{\mathbb{R}_+\times H}(\xi(\pi(x)),r).$$  Then we have that
\begin{equation*}
\begin{split}
\lim_{\substack{r<r_0,r\rightarrow r_0}}C^*(P_d(X),\mathcal {A}(H))_{O^1_r}^{\Gamma}+\lim_{\substack{r<r_0,r\rightarrow r_0}}C^*(P_d(X),\mathcal {A}(H))_{O^2_r}^{\Gamma}\\=\lim_{\substack{r<r_0,r\rightarrow r_0}}C^*(P_d(X),\mathcal {A}(H))_{O^1_r\cup O^2_r}^{\Gamma};
\end{split}
\end{equation*}
\begin{equation*}
\begin{split}
\lim_{\substack{r<r_0,r\rightarrow r_0}}C_L^*(P_d(X),\mathcal {A}(H))_{O^1_r}^{\Gamma}+\lim_{\substack{r<r_0,r\rightarrow r_0}}C_L^*(P_d(X),\mathcal {A}(H))_{O^2_r}^{\Gamma}\\=\lim_{\substack{r<r_0,r\rightarrow r_0}}C_L^*(P_d(X),\mathcal {A}(H))_{O^1_r\cup O^2_r}^{\Gamma};
\end{split}
\end{equation*}
\begin{equation*}
\begin{split}
\lim_{\substack{r<r_0,r\rightarrow r_0}}C^*(P_d(X),\mathcal {A}(H))_{O^1_r}^{\Gamma}\cap\lim_{\substack{r<r_0,r\rightarrow r_0}}C^*(P_d(X),\mathcal {A}(H))_{O^2_r}^{\Gamma}\\=\lim_{\substack{r<r_0,r\rightarrow r_0}}C^*(P_d(X),\mathcal{A}(H))_{O^1_r\cap O^2_r}^{\Gamma};
\end{split}
\end{equation*}
\begin{equation*}
\begin{split}
\lim_{\substack{r<r_0,r\rightarrow r_0}}C_L^*(P_d(X),\mathcal {A}(H))_{O^1_r}^{\Gamma}\cap\lim_{\substack{r<r_0,r\rightarrow r_0}}C_L^*(P_d(X),\mathcal {A}(H))_{O^2_r}^{\Gamma}\\=\lim_{\substack{r<r_0,r\rightarrow r_0}}C_L^*(P_d(X),\mathcal {A}(H))_{O^1_r\cap O^2_r}^{\Gamma}.
\end{split}
\end{equation*}
\end{lem}

\begin{proof}[\bf{Proof of Proposition \ref{twisted-cbc}}]
Since $$C^*(P_d(X),\mathcal {A}(H))^{\Gamma}=\lim\limits_{r\rightarrow +\infty} C^*(P_d(X),\mathcal {A}(H))_{O_r}^{\Gamma}$$
and $$C_L^*(P_d(X),\mathcal {A}(H))^{\Gamma}=\lim\limits_{r\rightarrow +\infty} C_L^*(P_d(X),\mathcal {A}(H))_{O_r}^{\Gamma},$$ where $O_r=\bigcup\limits_{1\leq k\leq k_0}O_{r,k}$, following the similar arguments in \cite{ShanWang} we obtain that
$$\lim\limits_{d\rightarrow \infty}\lim\limits_{r\rightarrow \infty}K_*( C^*(P_d(X),\mathcal {A}(H))_{O_r}^{\Gamma})\cong \lim\limits_{r\rightarrow \infty}\lim\limits_{d\rightarrow \infty}K_*( C^*(P_d(X),\mathcal {A}(H))_{O_r}^{\Gamma},$$
and
$$\lim\limits_{d\rightarrow \infty}\lim\limits_{r\rightarrow \infty}K_*(C_L^*(P_d(X),\mathcal {A}(H))_{O_r}^{\Gamma}\cong \lim\limits_{r\rightarrow \infty}\lim\limits_{d\rightarrow \infty}K_*(C_L^*(P_d(X),\mathcal {A}(H))_{O_r}^{\Gamma}.$$
We get the following commuting diagram:
$$\xymatrix{
 \lim\limits_{d\rightarrow\infty}K_*(C^{*}_{L}(P_{d}(X), \mathcal{A}(H))^\Gamma) \ar[d]_{\cong} \ar[r]^{e_*} & \lim \limits_{d\rightarrow\infty}K_*(C^{*}(P_{d}(X), \mathcal{A}(H))^\Gamma) \ar[d]_{\cong}  \\
 \lim \limits_{d\rightarrow\infty}\lim \limits_{r\rightarrow\infty}K_*(C_L^*(P_d(X),\mathcal {A}(H))_{O_r}^\Gamma) \ar[d]_{\cong}  \ar[r]^{e_*} & \lim \limits_{d\rightarrow\infty}\lim \limits_{r\rightarrow\infty}K_*(C^*(P_d(X),\mathcal {A}(H))_{O_r}^\Gamma)  \ar[d]_{\cong}  \\
   \lim \limits_{r\rightarrow\infty}\lim \limits_{d\rightarrow\infty}K_*(C_L^*(P_d(X),\mathcal {A}(H))_{O_r}^\Gamma) \ar[r]^{e_*} & \lim \limits_{r\rightarrow\infty}\lim \limits_{d\rightarrow\infty}K_*(C^*(P_d(X),\mathcal {A}(H))_{O_r}^\Gamma)  .}$$

The conclusion follows that the bottom horizon map is an isomorphism by using Lemma \ref{decomp-twisted-cbc}, Lemma \ref{algebra-m-v} and the Mayer--Vietoris sequence argument.
\end{proof}

\subsection{The maximal and reduced equivariant twisted Roe algebras}

In this subsection, we shall prove that the canonical quotient map $$\lambda: C_{max}^{*}(P_{d}(X), \mathcal{A}(H))^\Gamma \to C^{*}(P_{d}(X), \mathcal{A}(H))^\Gamma$$
is an isomorphism.

Recall that in Section \ref{twist-quotient}, we defined a $2$-sided $*$-ideal $C_{alg}^{*}(P_{d}(X), \mathcal{A}(H))_O^\Gamma$ of $C_{alg}^{*}(P_{d}(X), \mathcal{A}(H))^\Gamma$ consisting of the operators with entries supported in $O$ for each open subset $O$ in $\mathbb{R}_+\times H$, and the $C^*$-subalgebras $C_{max}^{*}(P_{d}(X), \mathcal{A}(H))_O^\Gamma$ and $C^{*}(P_{d}(X), \mathcal{A}(H))_O^\Gamma$ are the operator norm closure of $C_{alg}^{*}(P_{d}(X), \mathcal{A}(H))_O^\Gamma$ in $C_{max}^{*}(P_{d}(X), \mathcal{A}(H))^\Gamma$ and $C^{*}(P_{d}(X), \mathcal{A}(H))^\Gamma$, respectively. Recall that
$$O_{r,k}=\bigcup\limits_{x\in Y_k}B(\xi(\pi(x)), r)\subset \mathbb{R}_+\times H.$$

\begin{lem}\label{iso-idea-max-red-twi}
Let $X$ be a discrete metric space with bounded geometry, and let $\Gamma$ be a countable discrete amenable group. Assume that $\Gamma$ acts on $X$ properly and isometrically such that  the $\Gamma$-orbits are uniformly equivariantly coarsely equivalent. If the quotient space $X/\Gamma$ admits a coarse embedding into Hilbert space, then the canonical quotient map
$$\lambda:C_{max}^*(P_{d}(X),\mathcal {A})_{O_{r,k}}^{\Gamma}\rightarrow C^*(P_{d}(X),\mathcal {A}(H))_{O_{r,k}}^{\Gamma}$$
extended from the identity map on $C_{alg}^*(P_{d}(X),\mathcal {A}(H))_{O_{r,k}}^{\Gamma}$ is an isomorphism.
\end{lem}

\begin{proof}
For brevity, we denote by $O$ the open subset $O_{r,k}$ and denote $O_x=B(\xi(\pi(x)), r)$. Let $\Pi^u_{x\in Y_k}C^*_{alg}(P_{d}(X), \mathcal {A}(H))_{O_x}^{\Gamma}$ be the $*$-algebra consisting of the operators
$$\Pi_{x\in Y_k} T_x\in \Pi^u_{x\in Y_k}C^*_{alg}(P_d(X), \mathcal {A}(H))_{O_x}^{\Gamma},$$
such that the family of operators $\{T_x\}_{x\in Y_k}$ satisfies conditions in Definition \ref{alg-twist} with uniform constants.
Since $O=\bigsqcup_{x\in Y_k} O_x$, we have that
$$C^*_{alg}(P_{d}(X), \mathcal {A}(H))_O^{\Gamma}=\Pi^u_{x\in Y_k}C^*_{alg}(P_{d}(X), \mathcal {A}(H))_{O_x}^{\Gamma}.$$
By the definition of each $C^*_{alg}(P_{d}(X), \mathcal {A}(H))_{O_x}^{\Gamma}$, we have $$\Pi^u_{x\in Y_k}C^*_{alg}(P_{d}(X), \mathcal {A}(H))_{O_x}^{\Gamma}=\lim\limits_{R\rightarrow \infty}\Pi^u_{x\in Y_k}C^*_{alg}(N_R(\Gamma\cdot x), \mathcal {A}(H))_{O_x}^{\Gamma},$$
where the inductive limit can be viewed as the union of these $*$-algebras.

Denote by $|\Gamma|$ the metric space of $\Gamma$ endowed with a word-length metric. Let $\Pi^u_{x\in Y_k}C^*_{alg}(|\Gamma|, \mathcal {A}(H))_{O_x}^{\Gamma}$ be the $C^*$-algebra of all the operators
$$\Pi_{x \in Y_k} T_x \in \Pi^u_{x\in Y_k}C^*_{alg}(|\Gamma|, \mathcal {A}(H))_{O_x}^{\Gamma}$$
satisfying that
\begin{enumerate}[(1)]
    \item each $T_x$ is a $\Gamma\times \Gamma$-matrix and $T_x(g,h)=T_x(\gamma g, \gamma h)$ for $g,h,\gamma\in \Gamma$;
    \item there exists a constant $M>0$, such that $\sup_{x\in Y_k, g,h \in \Gamma}\|T_x(g,h)\|\leq M$;
    \item $\sup_{x \in Y}\mbox{propagation}(T_x)<\infty$;
    \item there exists $c>0$ such that if $Y=(s,v)\in \mathbb{R}_+\times W_N(\pi(x))$ with $\|Y\|=\sqrt{s^2+\|v\|^2}\leq 1$ and if $T'_x(g,h)$ is an element satisfying
$(\beta_N(\pi(x_x))\widehat{\otimes}1)(T'_x(g,h))=T(g,h)$, then the derivative of the
function $T'_x(g,h):\mathbb{R}\times W_N(\pi(x))\rightarrow \text{Cliff}(W_N(\pi(x)))\widehat{\otimes}K$ in the direction of $Y$ exists, denoted by $D_Y(T'_x(g,h))$, and $\|D_Y(T'_x(g,h))\|\leq c$ for all $g,h\in \Gamma$.
\end{enumerate}

Since the $\Gamma$-orbits are uniformly coarsely equivalent, all the metric spaces $N_R(\Gamma\cdot x)$ with $x\in Y_k$ are uniformly coarsely equivalent to $|\Gamma|$. It follows that
$$\lim\limits_{R\rightarrow \infty}\Pi^u_{x\in Y_k}C^*_{alg}(N_R(\Gamma\cdot x), \mathcal {A}(H))_{O_x}^{\Gamma}\cong \Pi^u_{x\in Y_k}C^*_{alg}(|\Gamma|, \mathcal {A}(H))_{O_x}^{\Gamma}.$$
Moreover, since
$$\Pi^u_{x\in Y_k}C^*_{alg}(|\Gamma|, \mathcal {A}(H))_{O_x}^{\Gamma}= C^*_{alg}(|\Gamma|, \Pi_{x\in Y_k}\mathcal {A}(H)_{O_x})^{\Gamma},$$
we have that
$$C^*_{max}(P_{d}(X), \mathcal {A}(H))_O^{\Gamma}\cong  C^*_{max}(|\Gamma|, \Pi_{j\in J_k}\mathcal {A}(H)_{O_x})^{\Gamma}$$
and $$C^*(P_{d}(X), \mathcal {A}(H))_O^{\Gamma}\cong  C^*(|\Gamma|, \Pi_{x\in Y_k}\mathcal {A}(H)_{O_x})^{\Gamma}.$$

Since $\Gamma$ is amenable, there exists a sequence of finitely supported positive type functions $\phi_n:\Gamma\rightarrow \mathbb{C}$ such that $\phi_n$ converge pointwise to $1$(c.f. \cite[Theorem~2.6.8]{BO}). Denote $$k_n(g,h)=\phi_n(g^{-1}h)$$
for all $g,h\in \Gamma$. It is obvious that each $k_n$ is a positive definite kernel $k_n:\Gamma\times \Gamma \to [0,1]$ such that
$$k_n(\gamma g, \gamma h)=k_n(g,h),$$
for all $\gamma, g, h \in \Gamma$.
Therefore, one can define a sequence of unital and completely positive Schur multipliers $$M^{alg}_{k_n}:C^*_{alg}(|\Gamma|, \Pi_{x\in Y_k}\mathcal {A}(H)_{O_x})^{\Gamma}\rightarrow C^*_{alg}(|\Gamma|, \Pi_{x\in Y_k}\mathcal {A}(H)_{O_x})^{\Gamma}$$
by
$$M^{alg}_{k_n}(T)(g,h)=k_n(g,h)T(g,h)$$
for all $T\in C^*_{alg}(|\Gamma|, \Pi_{x\in Y_k}\mathcal {A}(H)_{O_x})^{\Gamma}$ and $g, h\in \Gamma$.
Each operator $M^{alg}_{k_n}$ continuously extends to completely positive maps $$M^{max}_{k_n}:C^*_{max}(|\Gamma|, \Pi_{x\in Y_k}\mathcal {A}(H)_{O_x})^{\Gamma}\rightarrow C^*_{max}(|\Gamma|, \Pi_{x\in Y_k}\mathcal {A}(H)_{O_x})^{\Gamma}$$
and
$$M_{k_n}:C^*(|\Gamma|, \Pi_{x\in Y_k}\mathcal {A}(H)_{O_x})^{\Gamma}\rightarrow C^*(|\Gamma|, \Pi_{x\in Y_k}\mathcal {A}(H)_{O_x})^{\Gamma}$$
By the definition of $k_n$, we know that $M^{max}_{k_n}$ and $M_{k_n}$ converge in point-norm to identity maps on $C_{max}^*(|\Gamma|, \Pi_{x\in Y_k}\mathcal {A}(H)_{O_x})^{\Gamma}$ and $C^*(|\Gamma|, \Pi_{x\in Y_k}\mathcal {A}(H)_{O_x})^{\Gamma}$, respectively.

Now we have the following commutative diagram:
$$\xymatrix{
  C^*_{max}(|\Gamma|, \Pi_{x\in Y_k}\mathcal {A}(H)_{O_x})^{\Gamma} \ar[d]_{\lambda} \ar[r]^{M^{max}_{k_n}} & C^*_{max}(|\Gamma|, \Pi_{x\in Y_k}\mathcal {A}(H)_{O_x})^{\Gamma} \ar[d]^{\lambda} \\
  C^*(|\Gamma|, \Pi_{x\in Y_k}\mathcal {A}(H)_{O_x})^{\Gamma} \ar[r]^{M_{k_n}} & C^*(|\Gamma|, \Pi_{x\in Y_k}\mathcal {A}(H)_{O_x})^{\Gamma} }.$$
For each operator $T\in \text{Ker}(\lambda)$, we have $$M_{k_n}(\lambda(T))=\lambda(M^{max}_{k_n}(T))=0.$$
Since $M^{max}_{k_n}(T)\in C^*_{alg}(|\Gamma|, \Pi_{x\in Y_k}\mathcal {A}(H))_{O_x}^{\Gamma}$ and the restriction of $\lambda$ on the $*$-subalgebra $C^*_{alg}(|\Gamma|, \Pi_{x\in Y_k}\mathcal {A}(H))_{O_x}^{\Gamma}$ is the identity map, we have that $M^{max}_{k_n}(T)=0$. It follows that
$$0=T=\lim\limits_{n \to \infty}M^{max}_{k_n}(T).$$
Therefore, $\lambda$ is injective. As a result, the map
$$\lambda:C_{max}^*(P_{d}(X),\mathcal {A}(H))_{O_{r,k}}^{\Gamma}\rightarrow C^*(P_{d}(X),\mathcal {A}(H))_{O_{r,k}}^{\Gamma}$$
is an isomorphism.

\end{proof}

We shall glue all the isomorphisms between ideals together using the following result.
\begin{lem}\label{idea-additive}
Let $A,A'$ be $C^*$-algebras, $\varphi:A\rightarrow A'$ a $*$-homomorphism. Assume $A=I+J$, $A'=I'+J'$, where $I,J$ are 2-sided ideals of $A$, and $I', J'$ are 2-sided ideals of $A'$. If the restriction maps $\varphi_I:I\rightarrow I'$, $\varphi_J:J\rightarrow J'$ and $\varphi_{I\cap J}: I\cap J\rightarrow I'\cap J'$ are all isomorphisms, then $\varphi:A\rightarrow A'$ is an isomorphism.
\end{lem}

\begin{proof}
Since $A/I=I+J/I=J/I\cap J$, so the quotient map $\varphi_{A/I}:A/I\rightarrow A'/I'$ is an isomorphism.
We have the following commutative diagram
$$
\xymatrix{
  0\ar[r] & I\ar[d]^{\varphi_I}\ar[r] & A\ar[d]^\varphi\ar[r] & A/I\ar[d]^{\varphi_{A/I}}\ar[r] & 0\\
  0\ar[r] & I'\ar[r] & A'\ar[r] & A'/I'\ar[r] & 0.}
$$
As a consequence of diagram-chasing, we obtain that $\varphi$ is an isomorphism.
\end{proof}

\begin{prop}\label{iso-max-red-twist}
Let $X$ be a discrete metric space with bounded geometry, and $\Gamma$ a countable discrete group. Assume $\Gamma$ acts on $X$ properly and isometrically with the $\Gamma$-orbits uniformly equivariantly coarsely equivalent. If the quotient space $X/\Gamma$ is coarsely embeddable into Hilbert space and $\Gamma$ is amenable, then the homomorphism
$$\lambda:C_{max}^*(P_{d}(X),\mathcal {A}(H))^{\Gamma}\rightarrow C^*(P_{d}(X),\mathcal {A}(H))^{\Gamma}$$
is an isomorphism.
\end{prop}

\begin{proof}
 Denote
 $$O_r=\bigcup\limits_{\pi(x)\in X/\Gamma}B(\pi(x),r).$$
 Note that
 $C_{max}^*(P_{d}(X),\mathcal {A}(H))_{O_r}^{\Gamma}$ and $C^*(P_{d}(X),\mathcal {A}(H))_{O_r}^{\Gamma}$ are 2-sided ideals of $C_{max}^*(P_{d}(X),\mathcal {A}(H))^{\Gamma}$ and $C^*(P_{d}(X),\mathcal {A}(H))^{\Gamma}$, respectively. Moreover, we have that
 $$C_{max}^*(P_{d}(X),\mathcal {A}(H))^{\Gamma}=\lim\limits_{r\rightarrow +\infty} C_{max}^*(P_{d}(X),\mathcal {A}(H))_{O_r}^{\Gamma},$$
 and
 $$C^*(P_{d}(X),\mathcal {A}(H))^{\Gamma}=\lim\limits_{r\rightarrow +\infty} C^*(P_{d}(X),\mathcal {A}(H))_{O_r}^{\Gamma}.$$

For each $r>0$, $O_r=\bigcup\limits_{k=1}^{k_0} O_{r,k}$. By Lemma \ref{algebra-m-v}, \ref{iso-idea-max-red-twi}, \ref{idea-additive}, we have that the canonical quotient
 $$\lambda: C_{max}^*(P_{d}(X),\mathcal {A}(H))_{O_r}^{\Gamma}\to C^*(P_{d}(X),\mathcal {A}(H))_{O_r}^{\Gamma},$$
 extended from the identity map on $C_{alg}^*(P_{d}(X),\mathcal {A}(H))_{O_r}^{\Gamma}$ is an isomorphism.
 As a result, the homomorphism
$$\lambda:C_{max}^*(P_{d}(X),\mathcal {A}(H))^{\Gamma}\rightarrow C^*(P_{d}(X),\mathcal {A}(H))^{\Gamma}$$
is an isomorphism.
\end{proof}

\section{Proof of the main results}

In this section, we shall show that the maps
$$e_*: \lim\limits_{d \to \infty}K_*(C_L(P_d(X))^{\Gamma}) \to K_*(C^*(X)^{\Gamma})$$
and
$$\lambda_*: K_*(C_{max}^*(X)^{\Gamma})\to K_*(C^*(X)^{\Gamma})$$
are isomorphisms under the assumptions in Theorem 1.1. In order to prove the main result, we shall introduce a geometric analogue of Bott periodicity following the constructions of Yu \cite{Yu}.

Let $H$ be a separable infinite-dimensional real Hilbert space. Let $V\subseteq H$ be a finite
dimensional affine subspace of $H$. Denote by $V^0$ the finite-dimensional linear
subspace consisting of differences of elements in $V$. Let $\mathcal {L}^2(V)=L^2(V,{\rm Cliff}(V^0))$
be the graded infinite dimensional Hilbert space of $L^2$-maps from $V$ to the complexified Clifford algebra of $V^0$.

Let $V_a\subseteq V_b$ be finite-dimensional affine subspaces of $H$. Then we have an algebraic decomposition
$$V_{b}=V_{ba}^0\oplus V_a,$$
where $V_{ba}^0$ is the orthogonal complement of $V_a^0$ in $V_b^0$. Consider a unit vector $\xi_0\in \mathcal {L}^2(V_{ba}^0)$ defined by
$$\xi_0(w)=\pi^{-\frac{\text{dim}(V_{ba})}{4}}\exp (-\frac{1}{2}\|w\|^2),$$
for all $w \in V_{ba}^0$.
Then we can regard $\mathcal {L}^2(V_a)$ as a subspace of $\mathcal {L}^2(V_b)$ via an isometric inclusion given by
\begin{equation}
\label{sec:5:eq:identity}
i_{ba}:\mathcal {L}^2(V_{a})\rightarrow \mathcal {L}^2(V_{ba}^0)\widehat{\otimes}\mathcal {L}^2(V_{a})\cong \mathcal {L}^2(V_{b}),\quad\xi\mapsto \xi_0\widehat{\otimes}\xi.
\end{equation}
For any finite-dimensional affine subspaces $V_a\subseteq V_b\subseteq V_c$, we have $$i_{ca}=i_{cb}\circ i_{ba}.$$
Then we define $$\mathcal {L}^2(H)=\lim\limits_{\rightarrow}\mathcal {L}^2(V),$$
where the limit is taken over all the finite-dimensional affine subspaces $V\subseteq H$.

Let $\mathscr{S}(V)\subseteq \mathcal {L}^2(H)$ be the subspace of Schwartz class functions from $V$ to Cliff$(V^0)$. Choose an orthonormal basis $\{e_1,e_2,\cdots, e_n\}$ for $V^0$, and let $\{x_1, x_2, \cdots, x_n\}$ be the dual coordinates to $\{e_1,e_2,\cdots, e_n\}$, we define the Dirac operator $D_V$, an unbounded operator on $\mathcal {L}^2(H)$ with domain $\mathscr{S}(V)$, by the formula$$D_V\xi=\sum\limits_{i=1}^n (-1)^{\text{deg} \xi}\frac{\partial \xi}{\partial x_i}e_i,$$
where $e_i$ is the Clifford multiplication by $e_i\in V^0\subseteq \text{Cliff}(V^0)$. Define the Clifford operator $C_{V,v}$ with domain $\mathscr{S}(V)$, by $$(C_{V,v}\xi)(w)=(w-v)\cdot \xi(w),$$ where $v\in V$ is a fixed base point, and the multiplication is the Clifford multiplication by the vector $w-v\in V^0$.

Now fix $x\in P_d(X)$. Let $W_k(\pi(x))$ be the finite-dimensional subspace as in Definition \ref{supp_a}, and identify
$\mathcal {L}^2(W_k(\pi(x)))$ as a subspace of $\mathcal {L}^2(W_{k+1}(\pi(x)))$ via the isometric inclusion
$$\mathcal {L}^2(W_k(\pi(x)))\rightarrow \mathcal {L}^2(W_{k+1}(\pi(x)))$$
defined in (\ref{sec:5:eq:identity}). Note that these inclusions preserve the Schwartz subspaces $\mathscr{S}(W_k(\pi(x)))$. Define a Schwartz subspace of $\mathcal {L}^2(H)$ by taking the algebraic direct limit
$$\mathscr{S}(\pi(x))=\lim\limits_{\longrightarrow}\mathscr{S}(W_k(\pi(x))).$$

Let $V_0(\pi(x))=W_1(\pi(x))$ and $V_k(\pi(x))=W_{k+1}(\pi(x))\ominus W_k(\pi(x))$ if $k\geq 1$, where $x\in P_d(X)$. We consider the Dirac operator $D_k$ defined by $$D_k=D_{V_k(\pi(x))}$$ and Clifford operators $C_{k,\pi(x)}$ defined by
$$C_{0,\pi(x)}=C_{V_0(\pi(x)),f(\pi(x))};\; C_{k,\pi(x)}=C_{V_k(\pi(x)),0},\; k\geq 1,$$
associated to each $V_k(\pi(x))$. For each $n\in \mathbb{N}$ and $t\geq1$, define an unbounded operator
 $B_{n,t}(\pi(x))$ on $\mathcal {L}^2(H)$ by $$B_{n,t}(\pi(x))=\sum\limits_{k=0}^{n-1}(1+kt^{-1})D_k+\sum\limits_{k=n}^{\infty}(1+kt^{-1})(D_k+C_{k,\pi(x)})$$
associated to the decomposition $$V_0(\pi(x))\oplus V_1(\pi(x))\oplus\cdots\oplus V_n(\pi(x))\oplus\cdots.$$
Note that the operator $B_{n,t}(\pi(x))$ is well-defined on the Schwartz space $\mathscr {S}(\pi(x))$.

Let $\mathcal{K}$ be the algebra of compact operators on $\mathcal{H}=\mathcal {L}^2(H)$.
We can define the $*$-algebra $C_{alg}^*(X, \mathcal {K})^\Gamma$ similarly as $C^*(X)^\Gamma$ by changing the coefficients of $T\in C_{alg}^*(X)^\Gamma$ from $K$ to $\mathcal{K}\widehat{\otimes}K$. Let $C_{max}^*(X, \mathcal {K})^\Gamma$ and $C_{red}^*(X, \mathcal {K})^\Gamma$ be the maximal and reduced Roe algebra which are the completions of $C_{alg}^*(X, \mathcal {K})^\Gamma$ under the maximal and reduced norm, respectively.

Recall that in Section \ref{Sect-def-CBC}, we choose a $\Gamma$-invariant countable dense subset $X_d\subset P_d(X)$ for each $d>0$.
\begin{defn}\label{def-K-Roe-alg}
Let $C_{alg}^*(P_d(X), \mathcal {S}\widehat{\otimes}\mathcal {K})^\Gamma$ be the set of all functions $T$ from $X_d\times X_d$ to $\mathcal {S}\widehat{\otimes}\mathcal {K}\widehat{\otimes} K$ such that
\begin{enumerate}
\item[(1)] there exists $M>0$ such that $\|T(x,y)\|\leq M$ for any $(x,y)\in X_d\times X_d$;
\item[(2)] there exists $r_1>0$ such that $T(x,y)=0$ if $d(x,y)>r_1$;
\item[(3)] there exists $L>0$ such that for each $y\in X_d$, $$\sharp\{x:T(x,y)\neq 0\}\leq L,\qquad \sharp\{x:T(y,x)\neq 0\}\leq L;$$
\item[(4)] for any bounded set $B\subseteq P_d(X)$, the set $\{(x,y)\in B\times B\cap X_d\times X_d|T(x,y)\neq 0\}$ is finite;
\item[(5)] there exists $r_2>0$, such that if $|t|>r_2$, then $T(x,y)(t)=0$ for all $(x,y)\in X\times X$, where $T(x,y)\in {S}\widehat{\otimes}\mathcal {K}\widehat{\otimes} K$ can be viewed as a $\mathcal {K}\widehat{\otimes} K$ valued function on $\mathbb{R}$;
\item[(6)] there exists $c>0$ such that $\|\frac{d}{dt}T(x,y)\|\leq c$ for all $(x,y)\in X\times X$;
\item[(7)] $\gamma(T)=T$, i.e. $T(\gamma^{-1}x, \gamma^{-1}y)=T(x,y)$.
\end{enumerate}
\end{defn}

Following Definition \ref{alg-twist}, we can view $C_{alg}^*(P_d(X), \mathcal {S}\widehat{\otimes}\mathcal {K})^\Gamma$ as a $*$-algebra acting on the graded Hilbert module
$$\mathcal{E}=\ell^2(X_d)\widehat{\otimes}\ell^2(\Gamma)\widehat{\otimes} \mathcal{S}\widehat{\otimes}\mathcal{K}\widehat{\otimes} K $$ over $\mathcal {S}\widehat{\otimes}\mathcal {K}\widehat{\otimes} K$, where the $\mathbb{Z}_2$-grading on $\ell^2(X_d)$, $\ell^2(\Gamma)$ and $K$ are trivial and $\mathcal{S}$ is graded by the odd and even functions.

\begin{defn}
\begin{enumerate}[(1)]
    \item The $C^*$-algebra $C^*(P_d(X), \mathcal {S}\widehat{\otimes}\mathcal {K})^\Gamma$ is defined to be the operator norm closure of $C_{alg}^*(P_d(X), \mathcal {S}\widehat{\otimes}\mathcal {K})^\Gamma$ in $B(\mathcal{E})$, where $B(\mathcal{E})$ is the $C^*$-algebra of all adjointable module homomorphisms from $\mathcal{E}$ to $\mathcal{E}$.
    \item The $C^*$-algebra $C_{max}^*(P_d(X), \mathcal {S}\widehat{\otimes}\mathcal {K})^\Gamma$ is the completion of $C_{alg}^*(P_d(X), \mathcal {S}\widehat{\otimes}\mathcal {K})^\Gamma$ with respect to the $C^*$-norm
$$\|T\|_{max}:=\sup\{\|\phi(T)\|,\phi:C_{alg}^*(P_d(X), \mathcal {S}\widehat{\otimes}\mathcal {K})^\Gamma\rightarrow B(\mathcal{E}_{\phi}),\ \text{a}\ \text{$\ast$-representation}\},$$
where $\mathcal{E}_\phi$ is a Hilbert module over the $C^*$-algebra $\mathcal {S}\widehat{\otimes}\mathcal {K}\widehat{\otimes} K$.
\end{enumerate}
\end{defn}
Naturally, we have the canonical quotient map
$$
\lambda: C^*_{max}(P_d(X), \mathcal{S}\widehat{\otimes} \mathcal{K}) \to C^*(P_d(X), \mathcal{S}\widehat{\otimes} \mathcal{K})
$$
extended from the identity map on the $*$-algebra $C^*_{alg}(P_d(X), \mathcal{S}\widehat{\otimes} \mathcal{K})$.

We can similarly define the localization algebra, denoted $C^*_{L,alg}(P_d(X), \mathcal{S}\widehat{\otimes}\mathcal{K})^{\Gamma}$, to be the $*$-algebra of all uniformly bounded and uniformly continuous functions $f: [0,\infty)\to C^*_{alg}(P_d(X),\mathcal{S}\widehat{\otimes} \mathcal{K})^{\Gamma}$ such that
\begin{enumerate}[(1)]
    \item ${\rm propagation}(f(t))\to 0$ as $t \to \infty$,
    \item the constants in condition $3$, $4$, $5$ and $6$ of Definition \ref{def-K-Roe-alg} are independent on $t \in [0,\infty)$.
\end{enumerate}

\begin{defn}
The localization algebra, denoted by $C^*_L(P_d(X), \mathcal{K})^{\Gamma}$, is the completion of the $*$-algebra $C^*_{alg}(P_d(X), \mathcal{K})^{\Gamma}$ under the norm $$\|f\|=\sup_{t\geq 0}\|f(t)\|.$$
\end{defn}

From the definition of $C_{alg}^*(P_d(X), \mathcal {S}\widehat{\otimes}\mathcal {K})^\Gamma$ and the nuclearity of $\mathcal{S}$, we have
$$C^*(P_d(X), \mathcal {S}\widehat{\otimes}\mathcal {K})^\Gamma\cong \mathcal {S}\widehat{\otimes} C^*(P_d(X), \mathcal {K})^\Gamma,$$
$$C_L^*(P_d(X), \mathcal {S}\widehat{\otimes}\mathcal {K})^\Gamma\cong \mathcal {S}\widehat{\otimes} C_L^*(P_d(X), \mathcal {K})^\Gamma,$$
and
$$C_{max}^*(P_d(X), \mathcal {S}\widehat{\otimes}\mathcal {K})^\Gamma\cong \mathcal {S}\widehat{\otimes} C_{max}^*(P_d(X), \mathcal {K})^\Gamma.$$

For every non-negative integer $n$ and $x\in X$, we define
$$(\theta_t^n(x)):\mathcal {A}(W_n(\pi(x)))\widehat{\otimes} K\rightarrow \mathcal {S}\widehat{\otimes}K(\mathcal {L}^2(H))\widehat{\otimes} K$$
by
$$
(\theta_t^n(x))(g\widehat{\otimes}h\widehat{\otimes} k)=g_t(X\widehat{\otimes}1+1\widehat{\otimes}B_{n,t}(\pi(x)))(1\widehat{\otimes}M_{h_t})\widehat{\otimes}k$$
for all
$g\in \mathcal {S}$, $h\in C_0(W_n(\pi(x)), \text{Cliff}(W_n(\pi(x))))$, $k\in K$, $x\in X_d$, where $g_t(s)=g(t^{-1}s)$
for $s\in \mathbb{R}$, $h_t(v)=h(\xi(\pi(x))+t^{-1}(v-\xi(\pi(x))))$ for $v\in W_n(\pi(x))$, $M_{h_t}$ is the
pointwise multiplication operator on $\mathcal {L}^2(V)$ (where $W_n(\pi(x))\subset V$ is an affine
subspace of $H$) via the formula $$(M_{h_t}\xi)(v+w)=h_t(v)\xi(v+w)$$ for all $\xi\in \mathcal {L}^2(V)$ and $v\in W_n(\pi(x))$, $w\in V\ominus W_n(\pi(x))$. For every $T\in C_{alg}^*(X, \mathcal {\mathcal {A}})^\Gamma$, let $n$ be a non-negative integer such that for every $(x,y)\in X_d\times X_d$, there exists $T_1(x,y)\in \mathcal {A}(W_n(\pi(x)))\widehat{\otimes} K$ satisfying
$\beta_n(x)(T_1(x,y))=T(x,y)$. We define
$$(\alpha_t(T))(x,y)=(\theta_t^n(x))(T_1(x,y))$$ for
every $T\in C_{alg}^*(P_d(X), \mathcal {A})^\Gamma$, by Lemma 5.8 in \cite{HKT}, we know that
$(\alpha_t(T))(x,y)\in K(\mathcal{L}^2(H))\widehat{\otimes}K$.

\begin{defn}\label{dirac-map}
For each $t\geq 1$, we define $$\alpha_t:C_{alg}^*(P_d(X), \mathcal {A})^\Gamma\rightarrow C_{alg}^*(P_d(X), \mathcal {S}\widehat{\otimes}\mathcal {K})^\Gamma$$
via the formula $$(\alpha_t(T))(x,y)=(\theta_t^n(x))(T_1(x,y))$$ for every $T\in C_{alg}^*(P_d(X), \mathcal {A})^\Gamma$, where
$n$ is a non-negative integer such that for every $(x,y)\in X_d\times X_d$, there exists $T_1(x,y)\in \mathcal {A}(W_n(\pi(x)))\widehat{\otimes} K$ satisfying  $\beta_n(\pi(x))(T_1(x,y))=T(x,y)$.
\end{defn}

By the similar arguments as in Lemma 7.2 in \cite{Yu} and Proposition \ref{iso-max-red-twist}, we have the asymptotic morphisms
$$\alpha_{max,t}:C_{max}^*(P_d(X), \mathcal {A})^\Gamma\rightarrow C_{max}^*(P_d(X), \mathcal {S}\widehat{\otimes}\mathcal {K})^\Gamma\cong\mathcal {S}\widehat{\otimes}C_{max}^*(P_d(X), \mathcal {K})^\Gamma,$$
and
$$\alpha_t:C^*(P_d(X), \mathcal {A})^\Gamma\rightarrow C^*(P_d(X), \mathcal {S}\widehat{\otimes}\mathcal {K})^\Gamma\cong\mathcal {S}\widehat{\otimes}C^*(P_d(X), \mathcal {K})^\Gamma$$
for all $t\geq 1$.
Then we have the induced homomorphisms
$$\alpha_{max,*}:K_*(C_{max}^*(P_d(X),\mathcal {A})^\Gamma)\rightarrow K_*(\mathcal {S}\widehat{\otimes}C_{max}^*(P_d(X), \mathcal {K})^\Gamma),$$
and
$$\alpha_*:K_*(C^*(P_d(X),\mathcal {A})^\Gamma)\rightarrow K_*(\mathcal {S}\widehat{\otimes}C^*(P_d(X), \mathcal {K})^\Gamma)$$
on $K$-theory.

Moreovere, we can define an asymptotic morphism
$$\alpha_{L,t}:C_L^*(P_d(X), \mathcal {A})^\Gamma\rightarrow \mathcal {S}\widehat{\otimes}C_L^*(P_d(X), \mathcal {K})^\Gamma$$
by
$$\alpha_{L,t}(T_s)(t)=\alpha_t(T_s)$$
for all path $(T_s)_{s \in [0,\infty)}\in C_L^*(P_d(X), \mathcal {A})^\Gamma$ and all $t\geq 1$.
In addition, we have the homomorphism
$$\alpha_{L,*}:K_*(C_L^*(P_d(X),\mathcal {A})^\Gamma)\rightarrow K_*(\mathcal {S}\widehat{\otimes}C_L^*(P_d(X), \mathcal {K})^\Gamma)$$
induced by the asymptotic morphism $(\beta_{L,t})_{t \in [0, \infty)}$.

\begin{defn}
For each $t\geq 1$, define a map
$$\beta_t:\mathcal {S}\widehat{\otimes}C_{alg}^*(P_d(X))^\Gamma\rightarrow C_{alg}^*(P_d(X),\mathcal {A})^\Gamma$$ by$$(\beta_t(g\widehat{\otimes }T))(x,y)=(\beta(\pi(x)))(g_t)\widehat{\otimes}T(x,y)$$
for all $g\in \mathcal {S}$, $T\in C_{alg}^*(X)^\Gamma$ and $x,y\in X$, where $g_t(s)=g(t^{-1}s)$ for all $s\in \mathbb{R}$, and $\beta(\pi(x)):\mathcal {S}=\mathcal {A}(f(\pi(x)))\rightarrow \mathcal {A}(H)$ is the $*$-homomorphism defined in Section 3.
\end{defn}

Following the arguments in Lemma 7.6 in \cite{Yu}, we know that $\beta_t$ extends to an asymptotic morphisms
$$\beta_{max, t}: \mathcal {S}\widehat{\otimes}C_{max}^*(P_d(X))^\Gamma \rightarrow C_{max}^*(P_d(X), \mathcal {A})^\Gamma$$
and
$$\beta_t: \mathcal {S}\widehat{\otimes}C^*(P_d(X))^\Gamma \rightarrow C^*(P_d(X), \mathcal {A})^\Gamma$$
for all $t\geq 0$.
Then we have the induced homomorphism
$$\beta_{max,*}:K_*(\mathcal {S}\widehat{\otimes}C_{max}^*(P_d(X))^\Gamma)\rightarrow K_*(C_{max}^*(P_d(X), \mathcal {A})^\Gamma)$$
and
$$\beta_*:K_*(\mathcal {S}\widehat{\otimes}C^*(P_d(X))^\Gamma)\rightarrow K_*(C^*(P_d(X), \mathcal {A})^\Gamma)$$
on $K$-theory.

Applying the map $\beta_t:\mathcal {S}\widehat{\otimes}C_{alg}^*(P_d(X))^\Gamma\rightarrow C_{alg}^*(P_d(X),\mathcal {A})^\Gamma$ point-wise gives rise to an asymptotic morphism $$\beta_{L, t}: \mathcal {S}\widehat{\otimes}C_L^*(P_d(X))^\Gamma \rightarrow C_L^*(P_d(X), \mathcal {A})^\Gamma$$
for all $t\geq 1$.
Then we have a homomorphism
$$\beta_{L,*}:K_*(\mathcal {S}\widehat{\otimes}C_L^*(P_d(X))^\Gamma)\rightarrow K_*(C_L^*(P_d(X), \mathcal {A})^\Gamma)$$
induced by the asymptotic morphism $(\beta_{L,t})_{t \in [0,\infty)}$ on $K$-theory. Now we are ready to prove the geometric analogue of Bott periodicity.
\begin{prop}\label{identity-bott-dirac}
The compositions
 $$\alpha_*\circ\beta_*: K_*(\mathcal{S}\widehat{\otimes}C^*(P_d(X))^{\Gamma}) \to K_*(C^*(P_d(X), \mathcal{S}\widehat{\otimes}\mathcal{K})^{\Gamma}),$$
 $$\alpha_{max,*}\circ \beta_{max,*}:K_*(\mathcal{S}\widehat{\otimes}C_{max}^*(P_d(X))^{\Gamma}) \to K_*(C_{max}^*(P_d(X), \mathcal{S}\widehat{\otimes}\mathcal{K})^{\Gamma})$$
 and
 $$\alpha_{L,*}\circ \beta_{L,*}:K_*(\mathcal{S}\widehat{\otimes} C_L^*(P_d(X))^{\Gamma}) \to K_*(C_L^*(P_d(X), \mathcal{S}\widehat{\otimes}\mathcal{K})^{\Gamma})$$
 are identity maps, respectively.
\end{prop}
\begin{proof}
We shall only prove the first composition $\alpha_*\circ\beta_*$ is an identity map, and the other two can be proved similarly.

Define the asymptotic morphism $$\gamma:C_c(\mathbb{R})\widehat{\otimes}C_{alg}^*(P_d(X))^\Gamma\rightarrow C_{alg}^*(P_d(X),\mathcal{S}\widehat{\otimes}\mathcal {K})^\Gamma\cong C_c(\mathbb{R})\widehat{\otimes}C_{alg}^*(P_d(X), \mathcal{K})^\Gamma$$ by
$$\gamma_t(g\widehat{\otimes}T)(x,y)=g_{t^2}(X\widehat{\otimes}1+1\widehat{\otimes}B_{0,t}(\pi(x)))\widehat{\otimes}T(x,y),$$
where $g\in C_c(\mathbb{R})$ has compact support and continuously differentiable, $T\in C_{alg}^*(P_d(X))^\Gamma$, $t\geq 1$.

For any $x\in X$, we define $$\eta:\mathcal{A}(W_1(\pi(x)))\rightarrow \mathcal{K}(\mathcal{H})\widehat{\otimes} K$$
by $$(\eta_t(x))((g\widehat{\otimes}h)\widehat{\otimes}k)=g_{t^2}(X\widehat{\otimes}1+1\widehat{\otimes}B_{1,t}(\pi(x)))(1\widehat{\otimes}M_{h_{t^2}})
\widehat{\otimes}k$$ for each $g\in \mathcal{S}$, $h\in \text{Cliff}(W_1(\pi(x)))$, $k\in K$ and $t\geq 0$. Similarly, we have an asymptotic morphism
$$\gamma':C_c(\mathbb{R})\widehat{\otimes}C_{alg}^*(P_d(X))^\Gamma\rightarrow C_{alg}^*(P_d(X),\mathcal{S}\widehat{\otimes}\mathcal {K})^\Gamma\cong C_c(\mathbb{R})\widehat{\otimes}C_{alg}^*(P_d(X), \mathcal{K})^\Gamma$$ by
$$(\gamma'_t(g\widehat{\otimes}T))(x,y)=(\eta_t(x))(\beta_{W_1(\pi(x))}(g)\widehat{\otimes}T(x,y)),$$
where $g\in C_c(\mathbb{R})$ has compact support and continuously differentiable, $T\in C_{alg}^*(P_d(X))^\Gamma$, $t\geq 1$ and $\beta_{W_1(\pi(x))}:\mathcal{S}=\mathcal{A}(\pi(x))\rightarrow \mathcal{A}(W_1(\pi(x)))$ is the homomorphism induced by the map $g\mapsto g(X\widehat{\otimes} 1+1\widehat{\otimes} C_{W_1(\pi(x)),\pi(x)})$.

With the similar proof of Lemma 7.6 in \cite{Yu}, we have $$\gamma_t(g\widehat{\otimes}T)\leq \|g\|\|T\|, \gamma'_t(g\widehat{\otimes}T)\leq \|g\|\|T\|$$
for $g\in C_c(\mathbb(R))$ and $T\in C_{alg}^*(P_d(X))^\Gamma$, so $\gamma$ and $\gamma'$ can be extended to the asymptotic morphism $\gamma$ and $\gamma'$ from $\mathcal{S}\widehat{\otimes}C^*(P_d(X))^\Gamma$ to $ \mathcal{S}\widehat{\otimes}C^*(P_d(X), \mathcal{K})^\Gamma$ respectively. By the proof of Proposition 4.2 in \cite{HKT}, we know that $\gamma$ is asymptotically equivalent to $\gamma'$. Hence we have  $\gamma_*=\gamma'_*$ at the $K$-theory level.

Recall that $\mathcal{H}=\mathcal {L}^2(H)=\lim\limits_{\rightarrow}\mathcal {L}^2(V)$, where the limit is taken over all the finite-dimensional affine subspaces $V\subseteq H$. For any $x\in X$ and $t\geq 1$, let $U_x(t)$ be the unitary operator on $\mathcal{H}$ induced by the translation
$$v\mapsto v-t\xi(\pi(x))$$
on $V$. Let
$$
R(s)=\left(
  \begin{array}{cc}
    \cos\frac{\pi s}{2} & \sin\frac{\pi s}{2} \\
    -\sin\frac{\pi s}{2} & \cos\frac{\pi s}{2} \\
  \end{array}
\right), s\in [0,1].
$$
For each $x\in X$, $s\in [0,1]$ and $t\geq 1$, define a unitary operator $U_{x,s}(t)$ on $(l^2(X)\widehat{\otimes} \mathcal{H} \widehat{\otimes} H)\oplus(l^2(X)\widehat{\otimes} \mathcal{H} \widehat{\otimes} H)$ by
$$
U_{x,s}(t)=R(s)\left(
  \begin{array}{cc}
   1\widehat{\otimes}U_x(t)\widehat{\otimes}1  & 0 \\
    0 & 1 \\
  \end{array}
\right)R^{-1}(s).
$$
Now we define a family of asymptotic morphisms $$\gamma(s): C_c(\mathbb{R})\widehat{\otimes} C^*_{alg}(P_d(X))^\Gamma\longrightarrow  C_c(\mathbb{R})\widehat{\otimes}C_{alg}^*(P_d(X), \mathcal{K})^\Gamma\widehat{\otimes} M_2(\mathbb{C})$$
by $$\gamma_t(s)(g\widehat{\otimes}T)(x,y)=U_{x,s}(t)\left(
                                                       \begin{array}{cc}
                                                         (\gamma'_t(g\widehat{\otimes}T))(x,y) & 0 \\
                                                         0 & 0 \\
                                                       \end{array}
                                                     \right)U^{-1}_{x,s}(t)
$$
for $g\in C_c(\mathbb{R})$ has compact support and continuously differentiable, $T\in C_{alg}^*(P_d(X))^\Gamma$, $t\geq 1$, $s\in [0,1]$ and $(x,y)\in X_d\times X_d$.

Similarly, $\gamma(s)$ can be extended to the asymptotic morphisms $\gamma(s)$ for the reduced extensions.
And from the definition, it is easy to know that $\gamma(s)$ is a homotopy of asymptotic morphisms.

Let $g\in C_c(\mathbb{R})$ be a continuously differentiable function with compact support. For any $T\in C_{alg}^*(P_d(X))^\Gamma$, and $t\geq 1$, we have that
$$\gamma_t(0)(g\widehat{\otimes}T)-\left(
                                     \begin{array}{cc}
                                       \alpha_t(\beta_t(g\widehat{\otimes}T)) & 0 \\
                                       0 & 0 \\
                                     \end{array}
                                   \right)\rightarrow 0.
$$
As a result, we have
$$\gamma(0)_*=\alpha_*\circ \beta_*$$ at the $K$-theory level. On the other hand, $$
\gamma(1)=
\left(
                                                 \begin{array}{cc}
                                                   \gamma'_t & 0 \\
                                                   0 & 0 \\
                                                 \end{array}
                                               \right),
$$
so $\gamma(1)_*=\gamma'_*$. It follows that $\alpha_*\circ \beta_*=\gamma_*$.

Replacing $B_{0,t}(\pi(x))$ with $s^{-1}B_{0,t}(\pi(x))$ in the definition of $\gamma$ for all $s\in (0,1]$, we obtain a homotopy between $\gamma$ and the homomorphism induced by $g\widehat{\otimes}T\mapsto g\widehat{\otimes}(P\widehat{\otimes}T)$, where $P$ is a projection onto the one-dimensional kernel of $B_{0,t}(\pi(x))$ which does not depend on $x$. It turns out that $\gamma_*$ is an identity, so $\alpha_*\circ \beta_*$ is the identity.

\end{proof}
\par
Finally, we are ready to complete the proofs of the main results in this paper.
\par
\begin{proof}[{\bf Proof of Theorem \ref{main-thm}}]
We have the following commutative diagram
$$\xymatrix{
  \lim\limits_{d\rightarrow \infty}K_*(\mathcal {S}\widehat{\otimes}C_L^*(P_d(X))^{\Gamma}) \ar[d]_{\beta_{L,*}} \ar[r]^{e_*} & \lim\limits_{d\rightarrow \infty}K_*(\mathcal {S}\widehat{\otimes}C^*(P_d(X))^{\Gamma}) \ar[d]^{\beta_*} \\
  \lim\limits_{d\rightarrow \infty}K_*(C_L^*(P_d(X),\mathcal {A}(H))^{\Gamma}) \ar[d]_{\alpha_{L,*}} \ar[r]_{\cong}^{e_*} & \lim\limits_{d\rightarrow \infty}K_*(C^*(P_d(X),\mathcal {A}(H))^{\Gamma}) \ar[d]^{\alpha_*} \\
  \lim\limits_{d\rightarrow \infty}K_*(\mathcal {S}\widehat{\otimes}C_L^*(P_d(X))^{\Gamma})\ar[r]^{e_*} & \lim\limits_{d\rightarrow \infty}K_*(\mathcal {S}\widehat{\otimes}C^*(P_d(X))^{\Gamma}).   }$$

By Proposition \ref{identity-bott-dirac}, we know that $\alpha_*\circ\beta_*$ and $(\alpha_L)_*\circ(\beta_L)_*$ are identity maps, respectively.
It follows immediately from the diagram chasing that the map
$$e_*:\lim\limits_{d\rightarrow \infty }K_*(C_L^*(P_d(X))^{\Gamma})\rightarrow \lim\limits_{d\rightarrow \infty }K_*(C^*(P_d(X))^\Gamma)$$
induced by evaluation-at-zero map between the localization algebra and Roe algebra is an isomorphism on $K$-theory.
\end{proof}

\begin{proof}[{\bf Proof of Theorem \ref{main-thm-2}}]
For any $d>0$, we have the following commutative diagram
$$\xymatrix{
  K_*(\mathcal {S}\widehat{\otimes}C_{max}^*(P_d(X))^{\Gamma}) \ar[d]_{\beta_{max,*}} \ar[r]^{\lambda_*} & K_*(\mathcal {S}\widehat{\otimes}C^*(P_d(X))^{\Gamma}) \ar[d]^{\beta_*} \\
  K_*(C_{max}^*(P_d(X),\mathcal {A}(H))^{\Gamma}) \ar[d]_{\alpha_{max,*}} \ar[r]_{\cong}^{\lambda_*} & K_*(C^*(P_d(X),\mathcal {A}(H))^{\Gamma}) \ar[d]^{\alpha_*} \\
  K_*(\mathcal {S}\widehat{\otimes}C_{max}^*(P_d(X))^{\Gamma})\ar[r]^{\lambda_*} & K_*(\mathcal {S}\widehat{\otimes}C^*(P_d(X))^{\Gamma}).   }$$

By Proposition \ref{identity-bott-dirac}, we know that $\alpha_{max,*}\circ \beta_{max,*}$ and $\alpha_*\circ\beta_*$ are identity maps, respectively. By Theorem \ref{iso-max-red-twist} and the diagram chasing, we have that the map
$$\lambda_*:K_*(C_{max}^*(P_d(X))^{\Gamma})\rightarrow K_*(C^*(P_d(X))^{\Gamma})$$
induced by the canonical quotient is an isomorphism on $K$-theory. Consequently, we have that $$\lambda_*:K_*(C_{max}^*(X)^{\Gamma})\rightarrow K_*(C^*(X)^{\Gamma})$$ is an isomorphism since $P_d(X)$ is $\Gamma$-equivariantly coarsely equivalent to $X$ for a suitably large $d$.
\end{proof}

\par\noindent{\bf Acknowledgement.} The authors would like to thank Professor Guoliang Yu for helpful conversations. The first author thanks Professor Matthew Kennedy for his support.

\bigskip

\footnotesize

\noindent  Jintao Deng \\
Department of Mathematics,\\
University of Waterloo, Waterloo N2L3G1, Canada\\
E-mail: \url{jintao.deng@uwaterloo.ca}\\

\noindent  Benyin Fu \\
College of Statistics and Mathematics,\\
Shanghai Lixin University of Accounting and Finance,\\
Shanghai 201209, P. R. China.\\
E-mail: \url{fuby@lixin.edu.cn}\\

\noindent Qin Wang\\
Research Center for Operator Algebras, and Shanghai Key Laboratory of Pure Mathematics and Mathematical Practice, School of Mathematical Sciences,  East China Normal University,\\
Shanghai 200241, P. R. China.\\
Email: qwang@math.ecnu.edu.cn\\

\end{document}